\documentclass{amsart}
\usepackage{lmodern}
\usepackage[T5]{fontenc}
\usepackage{graphicx}
\usepackage{amssymb}
\usepackage{amsmath,amscd}
\usepackage{bbm}
\usepackage[v2,cmtip]{xy}
\usepackage{mathrsfs}
\usepackage{rotating}
\usepackage{hyperref}
\usepackage{todonotes}
\input diagxy

\newcommand{\colim}{\displaystyle{\lim_{\longrightarrow}}}

\def\f2{{\mathbb F}_{2}}
\def\z2{{\mathbb Z}/2}

\def\fp{{\mathbb F}_{p}}

\DeclareMathOperator{\Hom}{Hom}

\newcommand{\V}{\mathbb{V}}
\newtheorem{thm}{Theorem}
\newtheorem{lem}[thm]{Lemma}
\newtheorem{pro}[thm]{Proposition}
\newtheorem{cor}[thm]{Corollary}

\newtheorem{rem}[thm]{Remark}
\numberwithin{thm}{section}
\numberwithin{equation}{section}
\begin{document}
\title[Modular coinvariants and the mod $p$ homology of $QS^k$]{Modular coinvariants and the mod $p$ homology of $QS^k$}
\thanks{This research is funded by Vietnam National Foundation for Science and Technology Development (NAFOSTED) under grant number 101.04-2014.38.}
%\titlerunning{Modular coinvariants and the mod $p$ homology of $QS^k$}
\author{Phan Ho{\`a}ng Ch\horn{o}n}
%\keywords{Homology operations \and Dyer-Lashof algebra \and Modular invariant \and Infinite loop space \and Hopf rings}
\address{Department of Mathematics and Application, Saigon University,
 273 An Duong Vuong, District 5, Ho Chi Minh city, Vietnam.}
 \email{chonkh@gmail.com}
%}
%\authorrunning{Phan H. Ch\horn{o}n}

\maketitle
\begin{abstract}
We use modular invariant theory to establish a complete set of relations of the mod $p$ homology of $\{QS^k\}_{k\geq0}$, for $p$ odd, as a ring object in the category of coalgebras (also known as a coalgebraic ring or a Hopf ring). We also describe the action of the mod $p$ Dyer-Lashof algebra as well as the mod $p$ Steenrod algebra on the coalgebraic ring.
\end{abstract}

\section{Introduction}
Let $G^*(-)$ denote an unreduced multiplicative cohomology theory. Then, $G^*(-)$ can be represented unstably by a family of infinite loop spaces $G_n$ of its associated $\Omega$-spectrum; that is, $\Omega G_{n+1}\simeq G_n$ for each $n \in \mathbb{Z}$, and there is a natural isomorphism $G^n(X)\cong[X,G_n]$, where we denote by $[X,Y]$ the homotopy classes of unbased maps from $X$ to $Y$( see \cite{Brown1962} for example). The collection of these spaces $G_*=\{G_k\}_{k\in \mathbb{Z}}$ is considered as a graded ring space with the loop sum
$
m \colon G_n\times G_n\rightarrow G_n
$
and the composition product
$
\mu \colon G_k\times G_\ell\rightarrow G_{k+\ell}.
$
Therefore, the (ordinary) homology of $\{G_k\}_{k\in\mathbb{Z}}$, beside the usual addition and coproduct, possesses two extra operations, induced by $m$ and $\mu$, denoted by $\star$ and $\circ$, respectively. These operations make the homology of $\{G_k\}_{k\in\mathbb{Z}}$ a ring object in the category of graded coalgebras, which is also called a Hopf ring or coalgebraic ring in the literature.  

Hopf ring has become an important structural organization tool in the study homology of $\Omega$-spectrum as well as the unreduced generalized multiplicative cohomology theory. For example, the Hopf ring for complex cobordism $MU$ is studied by Ravenel-Wilson \cite{Ravenel.Wilson1977}, the Hopf ring for Morava $K$-theory is studied by Wilson \cite{Wilson1984} and for connective Morava $K$-theory by Kramer\cite{Kramer1991},  Boardman-Kramer-Wilson \cite{Boardman.et.al.1999}. Recently, the Hopf ring structure for $BP$ and $KO, KU$ are respectively investigated by Kashiwabara \cite{Kashiwabara1994},  Kashiwabara-Strickland-Turner \cite{Kashiwabara.et.al.1996} and Morton-Strickland \cite{Morton.Strickland2002}. We refer to the articles of Ravenel-Wilson \cite{Ravenel.Wilson1977},  Hunton-Turner \cite{Hun-Tur98} and the survey of Wilson \cite{wilson2000} (and the references therein) about the origin of Hopf ring and its importance in algebraic topology. 

Let $QS^k=\colim \; \Omega^n\Sigma^n S^k$ be the infinite loop space of the sphere $S^k$. The collection of spaces $\{QS^k\}_{k\geq0}$ forms an $\Omega$-spectrum, called the sphere spectrum, and  the mod $p$ (ordinary) homology   $\{H_* QS^k\}_{k\geq0}$ possesses a Hopf ring structure which is universal in the following sense. It is well known that all spectra are module spectra over the sphere spectrum,  the mod $p$ homology of any infinite loop space becomes an $H_*QS^0$-module or $\{H_*QS^k\}_{k\geq0}$-module object in the category of coalgebras (also called coalgebraic module in the literature).  
 
We will work over an arbitrary but fixed odd prime $p$. Denote by $A,R$ the mod $p$ Steenrod algebra and mod $p$ Dyer-Lashof algebra respectively. It is known, see  Kashiwabara \cite{Kashiwabara2010b} for example, that the mod $p$ homology of any infinite loop space is an $A$-$H_*QS^0$-coalgebraic module. Further, according to May \cite{Coh-Lad-May76},  it also has a so-called $A$-$R$-allowable Hopf algebra, that is, a Hopf algebra on which both the Steenrod algebra and the Dyer-Lashof algebra act and the two actions satisfy some compatibility conditions. Thus, understanding the coalgebraic ring structure of $H_*QS^0$ plays an important role in the study of the homology of infinite loop spaces as well as in the study of the category of $A$-$H_*QS^0$-coalgebraic modules, the category of $A$-$R$-allowable Hopf algebras, and their relationship.

The aim of this paper is to give a complete description of the  Hopf ring for $\{H_*QS^k\}_{k\geq 0}$. Such a description has been given for $p=2$ by Turner \cite{Tur97} and Eccles-Turner-Wilson \cite{Eccles1997}. For odd primes, some information about the Hopf ring generators were obtained by YanFei Li \cite{Li1996} in his unpublished thesis. Our approach follows closely the strategy used by Turner \cite{Tur97} at $p=2$, where the main technical point is a careful analysis of the Dyer-Lashof algebra in terms of modular (co)invariant theory.

Let $[1]\in H_0 QS^0$ denote the image of the non-base point generator of $H_0S^0$ under the homomorphism induced by the obvious inclusion $S^0\hookrightarrow QS^0$ and let $\sigma$ denote the image of the generator of $H_1S^1$ under the homomorphism induced by the inclusion $S^1\hookrightarrow QS^1$. It follows from work of Araki-Kudo \cite{AK56}, Dyer-Lashof \cite{Dye-Las62} and May \cite{Coh-Lad-May76} that the mod $p$ homology of $\{QS^k\}_{k\geq0}$ is generated as a Hopf ring by the elements $Q^i[1], i\geq0,$ $\sigma$ (for $p=2$) and by $Q^i[1],i\geq 0$, $\beta Q^i[1], i\geq1$, $\sigma$ (for $p$ odd), where $Q^i$ denote the Dyer-Lashof operations \cite{Dye-Las62}. This actually corresponds to the fact that the Quillen's approximation map of the symmetric group by elementary abelian subgroups is a monomorphism \cite{Quillen1971.ann}. However, for a long time, no one undertook to study the relations until the importance of the coalgebraic ring structure of $H_*QS^k$ is clearly made again from works of Hunton-Turner \cite{Hun-Tur98} and Kashiwabara \cite{Kashiwabara2000} (which develop the homological algebra for the category of modules over a Hopf ring). These works are maybe the main motivation for the study in \cite{Tur97} and \cite{Eccles1997}, which give a description of a complete set of relations as a Hopf ring of $H_*QS^k$ for $p=2$. Later, it was discovered in \cite{Kashiwabara1995} that  a complete set of relations can be obtained from the fact that the Quillen map for the symmetric groups is actually an isomorphism at the prime $2$ (see \cite{Guna-Lannes-Zarati1989}). At odd primes, the Quillen map is no longer an isomorphism \cite{Guna-Lannes-Zarati1989}, thereby making it difficult to generalize the results in \cite{Tur97,Eccles1997}. However, when one restricts on the Bockstein-nil submodules, the Quillen's map of the symmetric groups is also an isomorphism \cite{Ha.Lesh2004} and this fact essentially allows Kashiwabara \cite{Kashiwabara2012} to to generalize the results in \cite{Tur97,Eccles1997} for the Bockstein-nil homology of $H_*QS^k$.  
  
Let $R[n]$ denote the subspace of the Dyer-Lashof algebra spanned by all monomials of length $n$.  Let $\V_n$ denote an elementary abelian $p$-group of rank $n$. The symmetric group $\Sigma_{p^n}$ acts on $\V_n$ by permuting its $p^n$ elements. Since $\V_n$ also acts on itself by translations,  there is a canonical inclusion $\V_n \rightarrow \Sigma_{p^n}$. Let $\mathscr{B}[n]$ denote the image of the induced map in cohomology of this inclusion.  May \cite[II.3]{Coh-Lad-May76}, M\`ui \cite{Mui75} (see also Kechagias \cite{Kechagias1994}) proved that there is an isomorphism of algebras between the dual of $R[n]$ and $\mathscr{B}[n]$. This isomorphism is the key result needed to establish a complete set of relations for the Hopf ring $\{H_*QS^k\}_{k\geq0}$ for $p$ odd. 

However, one major difference from the case $p=2$ is that $\mathscr{B}[n]$ is no longer the full invariant ring of $H^* B\V_n$ under the action of the general linear group $GL_n$. To overcome this difficulty, we will work with homology instead, and proceed by first constructing a new basis for $\mathscr{B}[n]^*$. Using this new basis and combining with the fact that the induced in map in homology of the Kahn-Priddy transfer, $tr_*^{(n)}$, is multiplicative with respect to the circle product in $H_*QS^0$ and is $GL_n$-invariant, we obtain an analogous description of a complete set of relation of $\{H_*QS^k\}_{k\geq0}$ as a coalgebraic ring for odd primes. This fact again confirms the close  correspondence between the Hopf ring structure of $\{H_*QS^k\}_{k\geq0}$ and the Quillen map for the symmetric groups. 

We would like to point out that the results in \cite{Tur97,Eccles1997} as well as in \cite{Kashiwabara2012} can be deduced from our description of the Hopf ring for $\{H_*QS^k\}_{k\geq0}$ simply by letting $p=2$ or by killing the action of the Bockstein operation. Some relations in the Hopf ring, with suitable modification, are analogous to those avalable at $p=2$. For example, the relations \eqref{eq:relation E(s) 2}-\eqref{eq: relation E(s) 3} in Proposition \ref{pro:relation of E(s)} can be derived from the multiplicativity and the $GL_2$-invariant of $tr_*^{(2)}$ as in the case $p=2$. However, the relation \eqref{eq:relation about excess_sigma} is completely new because the Boskstein operation plays significant role. We explain this in more details in Remark \ref{rem:4.8}. 

Finally, to fully understand the structure of the Hopf ring $\{H_* QS^k \}_{k \geq 0}$,  we also provide formulas for the action of the Steenrod algebra and of the Dyer-Lashof algebra on the Hopf ring in question.  
 
 Before describing the structure of this article, we would like to point out that one possible direct application of our result is to provide another description of the mod $p$ homology of the finite symmetric groups. It has long been known to topologists that infinite loop spaces are built from symmetric groups: denote by $\Sigma_{\infty} = \colim \, \Sigma_n$ the direct limit of $\Sigma_n$ via natural inclusion $\Sigma_n \subset \Sigma_{n+1}$ as the permutations fixing $(n+1)$. The natural pairing $\Sigma_n \times \Sigma_m \rightarrow \Sigma_{m+n}$ yields a product in $H_*B\Sigma_{\infty}$. A theorem of fundamental importance in algebraic topology, due to Barratt-Priddy and Quillen \cite{Quillen1969b,Priddy1971,Barratt.Priddy1972}\footnote{According to Madsen and Milgram \cite[p. 49]{Madsen.Milgram1979}, Dyer-Lashof were the first to prove this result in an unpublished version of \cite{Dye-Las62}.}, states that there is a natural map 
$
\mathbb{Z} \times B\Sigma_{\infty}\to QS^0, 
$
inducing an isomorphism in homology (even integrally). Furthermore, it is additive with respect to the the addition in $\mathbb{Z}$ times the product in $H_* B\Sigma_{\infty}$ mentioned above and the loop sum product. Restricting to the component $Q_0 S^0$ of the constant loop, the homology isomorphism $B\Sigma_{\infty}\to Q_0S^0$ can be described explicitly (see for example \cite[p. 108]{Kahn.Priddy1978a}). In fact, a weight grading can be given on $H_* Q_0S^0$ so that those of weight at most $m$ are exactly the elements coming from $H_* B\Sigma_m$. From our point of view, 
$H_* B\Sigma_{p^n}$ is generated by those of the form
\[
E_{(\epsilon_1,i_1+b(I))}\circ\cdots\circ E_{(\epsilon_s,p^{s-1}(i_1+\cdots+i_s+b(I))-\Delta_s\epsilon_s)}\star[-p^s],
\]  
where $s \leq n$ (see Theorem \ref{thm:describe though dickson coninvariant}). The weight of this element, in the sense of Kahn and Priddy, is easily seen to be $p^s$. 
For general $m$, one uses its $p$-adic decomposition  to see that $H_* B\Sigma_m$ is generated by the $\star$-product of the elements of the form described above, having total weight at most $m$. The homology of $\Sigma_{p^n}$ is detected by the homology of the two subgroups $\V_n$ and $(\Sigma_{p^{n-1}})^p$. The relation \eqref{eq:relation about excess_sigma} occurs when a homology class is detected by both subgroups.

 The paper is divided into five sections. The first two sections are preliminaries. In Section \ref{sec:Preliminaries}, we review  the Dickson-M\`ui algebra; its relation with the image of the restriction from the cohomology of the symmetric group $\Sigma_{p^n}$ to the cohomology of the elementary abelian $p$-group of rank $n$, $\V_n$,  and with  the mod $p$ Dyer-Lashof algebra. In Section \ref{sec:Basis}, we construct new additive base for $(H^*B\V_n)^{GL_n}$, $(H_*B\V_n)_{GL_n}$, the image $\mathscr{B}[n]$ and the cokernel of the restriction map $H^*B\Sigma_{p^n} \rightarrow (H^*B\V_n)^{GL_n}$.  This calculation is of interest to topologists because May \cite{Coh-Lad-May76} have shown that the new basis of the dual of  $\mathscr{B}[n]$ can be  considered as an additive basis for the subspace $R[n]$ of the mod $p$ Dyer-Lashof algebra. The main result of this paper, which is a complete description of the Hopf ring for $\{H_*QS^k\}_{k\geq0}$ in terms of generators and relations is presented in Section \ref{sec: Hopf ring}. The final section is devoted for the desciption of the action of the Steenrod algebra and the Dyer-Lashof algebra on $\{H_*QS^k\}_{k\geq0}$. 
 
 Unless stated otherwise, we will be working over the prime order field $\fp$, where $p$ is an odd prime. 
\section{Preliminaries}\label{sec:Preliminaries}
In this section, we review some main points of the Dickson-M\`ui algebra and the image of the restriction from the cohomology of the symmetric group $\Sigma_{p^n}$ to the cohomology of the elementary abelian $p$-group of rank $n$. We also review some basic properties of the mod $p$ Dyer-Lashof algebra.
\subsection{Modular invariant}
Let $\V_n=(\mathbb{Z}/p\mathbb{Z})^n$ denote the rank $n$ elementary abelian $p$-group.
It is well-known that the mod $p$ cohomology of the classifying space $B\V_n$ is given by
\[
H^*B\V_n=E(e_1,\dots,e_n)\otimes \fp[x_1,\dots,x_n],
\]
where $(e_1,\dots,e_n)$ is a basis of $H^1B\V_n=\Hom(\V_n,\fp)$, $x_i=\beta(e_i)\in H^2B\V_n$ for $1\leq i\leq n$ with $\beta$ the homology Bockstein operation, $E(e_1,\dots,e_n)$ is the exterior algebra generated by $e_i$'s and $\fp[x_1,\dots,x_n]$ is the polynomial algebra generated by $x_i$'s.

Let $GL_n$ denote the general linear group $GL_n=GL(\V_n)$. The group $GL_n$ acts canonically on $\V_n$ and hence on $H^*B\V_n$ according to the following standard action
\[
(a_{ij})x_s=\sum_{i}a_{is}x_i,\quad (a_{ij})e_s=\sum_ia_{is} e_i, \quad (a_{ij}) \in GL_n. 
\]
The algebra of all invariants of $H^*B\V_n$ under the actions of $GL_n$ is computed by Dickson \cite{Dic11} at $p=2$ and M\`ui \cite{Mui75} for $p$ odd. 
We briefly summarize their results. For any $n$-tuple of non-negative integers $(r_1,\dots,r_n)$, put $[r_1,\dots,r_n]:=\det(x_i^{p^{r_j}})$, and define
\[
L_{n,i}:=[0,\dots,\hat{i},\dots,n]; \quad L_{n}:=L_{n,n};\quad q_{n,i}:=L_{n,i}/L_{n},
\]
for any $1 \leq i \leq n$. 
In particular, $q_{n,n}=1$ and by convention, set $q_{n,i}=0$ for $i<0$. The degree of $q_{n,i}$ is $2(p^n-p^i)$. Define
\[
V_n:=V_n(x_1,\dots,x_n):=\prod_{\lambda_j\in\fp}(\lambda_1x_1+\cdots+\lambda_{n-1}x_{n-1}+x_n).
\]

In fact, it can be shown that $V_n=L_{n}/L_{n-1}$ and $q_{n,i}$ can be inductively defined by the formula
\[
q_{n,i}=q_{n-1,i-1}^p+q_{n-1,i}V_n^{p-1}.
\] 

For non-negative integers $k, r_{k+1},\dots,r_n$, set
\[
[k;r_{k+1},\dots,r_n]:=\frac{1}{k!}
\left|
\begin{array}{ccc}
e_1&\cdots&e_n\\
\cdot&\cdots&\cdot\\
e_1&\cdots&e_n\\
x_1^{p^{r_{k+1}}}&\cdots&x_n^{p^{r_{k+1}}}\\
\cdot&\cdots&\cdot\\
x_1^{p^{r_{n}}}&\cdots&x_n^{p^{r_{n}}}
\end{array}
\right|.
\]
For $0\leq i_1<\cdots< i_k\leq n-1$, we define
\begin{align*}
M_{n;i_1,\dots,i_k}&:=[k;0,\dots, \hat{i}_1,\dots,\hat{i}_k,\dots,n-1],\\
R_{n;i_1,\dots,i_k}&:=M_{n;i_1,\dots,i_k}L_{n}^{p-2}.
\end{align*}

Note that $M_{n;i_1,\dots,i_k}$ is in degree $k+2((1+\cdots+p^{n-1})-(p^{i_1}+\cdots+p^{i_k}))$ and then the degree of $R_{n;i_1,\dots,i_k}$ is $k+2(p-1)(1+\cdots+p^{n-1})-2(p^{i_1}+\cdots+p^{i_k})$.

We put $P_n:=\fp[x_1,\dots,x_n]$. The subspace of all invariants of $H^*B\V_n$ under the action of $GL_n$ is given by the following theorem.

\begin{thm}[(Dickson \cite{Dic11}, M\`ui \cite{Mui75})] \label{thm:invariant of G}
\begin{enumerate}
\item The subspace of all invariants of $P_n$ under the action of $GL_{n}$ is given by
\[
D[n]:=P_n^{GL_n}=\fp[q_{n,0},\dots,q_{n,n-1}].
\]
\item As a $D[n]$-module, $(H^*B\V_n)^{GL_n}$ is free and has a basis consisting of $1$ and all elements of $\{R_{n;i_1,\dots,i_k}:1\leq k\leq n, 0\leq i_1<\cdots<i_k\leq n-1 \}$. In other words,
\begin{align*}
(H^*&B\V_n)^{GL_n}=
P_n^{GL_n}\oplus\bigoplus_{k=1}^n\bigoplus_{0\leq i_1<\cdots<i_k\leq n-1}R_{n;i_1,\dots,i_k}P_n^{GL_n}.
\end{align*}
\item The algebraic relations are given by
\begin{align*}
R_{n;i}^2&=0,\\
R_{n;i_1}\cdots R_{n;i_k}&=(-1)^{k(k-1)/2}R_{n;i_1,\dots,i_k}q_{n,0}^{k-1}
\end{align*}
for $0\leq i_1<\cdots <i_k<n$.
\end{enumerate}
\end{thm}

Part (i) of Theorem \ref{thm:invariant of G} can be viewed as a special case of the odd primary case. When $p=2$, it is well-known that $H^*B\V_n\cong \f2[e_1,\dots,e_n]$, therefore, the algebra of invariants $(H^*B\V_n)^{GL_n}$ can be deduced from part (ii) by halving the degree of elements. Here ``halving the degree of elements'' means the image under the map sending elements of even degree $2d$ to itself of degree $d$ and elements of odd degree to zero. Then, the relation (iii) becomes trivial.
\subsection{The Dyer-Lashof algebra}
Let us recall the construction of the Dyer-Lashof algebra. Let $\mathcal{F}$ be the free algebra generated by $\{f^i|i\geq 0\}$ and $\{\beta f^i|i>0\}$ over $\fp$, with $|f^i|=2i(p-1)$ and $|\beta f^i|=2i(p-1)-1$. Then $\mathcal{F}$ becomes a coalgebra equipped with coproduct $\psi:\mathcal{F}\rightarrow \mathcal{F}\otimes \mathcal{F}$ given by
\[
\psi (f^i)=\sum_{j=0}^i f^{i-j}\otimes f^j; \quad \psi (\beta f^{i+1})=\sum_{j=0}^i(\beta f^{i-j+1}\otimes f^j+f^{i-j}\otimes \beta f^{j+1}).
\]

In fact, $\mathcal{F}$ is a Hopf algebra with unit $\eta:\fp\rightarrow \mathcal{F}$ and augmentation $\varepsilon \colon \mathcal{F}\rightarrow \fp$ sending $f^0$ to $1$ and all other generators to zero.

A typical element of $\mathcal{F}$ has the form
\[
f^{\varepsilon,I}=\beta^{\epsilon_1}f^{i_1}\cdots\beta^{\epsilon_n}f^{i_n},
\]
where $(\varepsilon,I)=(\epsilon_1,i_1,\dots,\epsilon_n,i_n)$ with $\epsilon_j\in\{0,1\}$ and $i_j\geq \epsilon_j$ for  $1\leq j\leq n$. The degree of $f^{\varepsilon,I}$ is equal to $2(p-1)(i_1+\cdots+i_n)-(\epsilon_1+\cdots+\epsilon_n)$. Let $\ell(f^{\varepsilon,I})=n$ denote the length of $(\varepsilon,I)$ or $f^{\varepsilon,I}$ and let  the excess of $(\varepsilon,I)$ or $f^{\varepsilon,I}$ be denoted and defined by $exc(f^{\varepsilon,I})=2i_1-\epsilon_1-|f^{\varepsilon',I'}|$, where $(\varepsilon',I')=(\epsilon_2, i_2,\dots,\epsilon_n, i_n)$. In other words,
\begin{equation}\label{eq:excess of f}
\begin{split}
exc(f^{\varepsilon,I})&=2i_1-\epsilon_1-2(p-1)\sum_{k=2}^ni_j+\sum_{k=2}^n\epsilon_j\\
&=\sum_{k=1}^n2(i_{k}-pi_{k+1})-2\epsilon_1+b(I),
\end{split}
\end{equation}
where $b(I)=\sum_{k=1}^n\epsilon_k$ and $i_{n+1}=0$.
By convention, if $(\varepsilon,I)=\emptyset$, then its excess is $\infty$ and we omit $\epsilon_j$ if it is $0$. It is clear that the excess of $f^{\varepsilon,I}$ is non-negative if and only if that of $f^{\varepsilon_t,I_t}$ is non-negative for all $1\leq t\leq n$, where $(\varepsilon_t,I_t)=(\epsilon_t,i_t,\dots,\epsilon_n, i_n)$.

Let $T=\mathcal{F}/I_{exc}$, where $I_{exc}$ is the two-sided ideal of $\mathcal{F}$ generated by all elements of negative excess. Then $T$ inherits the structure of a Hopf algebra. Denote the image of $f^{\varepsilon,I}$ by $e^{\varepsilon,I}$. The degree, length, excess described above passes to $T$.

Let $I_{Adem}$ be the two-sided ideal of $T$ generated by Adem elements
\[
e^re^s-\sum_i(-1)^{r+i}\binom{(p-1)(i-s)-1}{pi-r}e^{r+s-i}e^i, r>ps;
\]
\begin{align*}
e^r\beta e^s-\sum_i(-1)^{r+i}&\binom{(p-1)(i-s)}{pi-r}\beta e^{r+s-i}e^i\\
+\sum_i&(-1)^{r+i}\binom{(p-1)(i-s)-1}{pi-r-1}e^{r+s-i}\beta e^i, r\geq ps.
\end{align*}
The quotient algebra $R=T/I_{Adem}$ is called the Dyer-Lashof algebra. We denote the image of $e^{\varepsilon,I}$ by $Q^{\varepsilon,I}$, then $Q^i$ and $\beta Q^i$ satisfy the Adem relations:
\begin{equation}\label{eq:adem relation 1}
Q^rQ^s=\sum_i(-1)^{r+i}\binom{(p-1)(i-s)-1}{pi-r}Q^{r+s-i}Q^i, r>ps;
\end{equation}
\begin{equation}\label{eq:adem relation 2}
\begin{split}
Q^r\beta Q^s=\sum_i&(-1)^{r+i}\binom{(p-1)(i-s)}{pi-r}\beta Q^{r+s-i}Q^i\\
-&\sum_i(-1)^{r+i}\binom{(p-1)(i-s)-1}{pi-r-1}Q^{r+s-i}\beta Q^i, r\geq ps.
\end{split}
\end{equation}

Let $P_*^r$ be the dual to the Steenrod cohomology operation $P^r$, then the Nishida relations hold:
\[
P_*^rQ^s=\sum_i(-1)^{r+i}\binom{(p-1)(s-r)}{r-pi}Q^{s-r+i}P_*^i;
\]
\begin{align*}
P_*^r\beta Q^s=\sum_i(-1)^{r+i}&\binom{(p-1)(s-r)-1}{r-pi}\beta Q^{s-r+i}P_*^i\\
&+\sum_i(-1)^{r+i}\binom{(p-1)(s-r)-1}{r-pi-1}Q^{s-r+i}P_*^i\beta.
\end{align*}

A monomials $Q^{\varepsilon,I}$ is called admissible if $(\varepsilon,I)$ is admissible (i.e. a string $(\varepsilon,I)=(\epsilon_1,i_1,\dots,\epsilon_n,i_n)$ is admissible if $pi_k-\epsilon_k\geq i_{k-1}$ for $2\leq k\leq n$).

Let $R[n]$ be the subspace of $R$ spanned by all monomials of length $n$. It has an additive basis consisting of all admissible monomials of length $n$ and non-negative excess. We will need a description of the dual $R[n]^*$ at $p$ odd, following May \cite{Coh-Lad-May76} and Kechagias \cite{Kechagias2004}.

For convenience we shall write $I$ instead of $(\varepsilon,I)$.

Let $I_{n,i}, J_{n;i}, K_{n;s,i}$ be admissible sequences of non-negative excess and length $n$ as follows
{\allowdisplaybreaks
\begin{align*}
I_{n,i}&=(p^{i-1}(p^{n-i}-1),\dots,p^{n-i}-1,p^{n-i-1},\dots,1);\\
J_{n;i}&=(p^{i-1}(p^{n-i}-1),\dots,p^{n-i}-1,(1,p^{n-i-1}),\dots,1);\\
K_{n;s,i}&=(p^{i-1}(p^{n-i}-1)-p^{s-1},\dots,p^{i-s}(p^{n-i}-1)-1), \\
&\quad(1,p^{i-s-1}(p^{n-i}-1)),p^{i-s-2}(p^{n-i}-1),\dots, p(p^{n-i}-1),\\
&\quad\quad\quad\quad p^{n-i}-1,(1,p^{n-i-1}),\dots,1).
\end{align*}
}
Then the excess of $Q^{I_{n,i}}$ is $0$ if $0<i\leq n-1 $ and $2$ if $i=0$; and
\begin{align*}
exc(Q^{J_{n;i}})&=1, 0\leq i\leq n-1;\\
exc(Q^{K_{n;s,i}})&=0, 0\leq s< i\leq n-1.
\end{align*}

Let 
$\xi_{n,i}=(Q^{I_{n,i}})^*, 0\leq i\leq n-1$, $\tau_{n;i}=(Q^{J_{n;i}})^*, 0\leq i\leq n-1,$ and $\sigma_{n;s,i}=(Q^{K_{n;s,i}})^*, 0\leq s<i\leq n-1,$
with respect to the admissible basis of $R[n]$.

The following theorem gives the structure of the dual of the Dyer-Lashof algebra.
\begin{thm}[(May \cite{Coh-Lad-May76}, see also Kechagias \cite{Kechagias2004})]\label{thm: relation of DL}
As an algebra, $R[n]^*$ is isomorphic to the free associative commutative algebra over $\fp$ generated by the set $\{\xi_{n,i},\tau_{n;i}, \sigma_{n;s,i}: 0\leq i\leq n-1, 0\leq s<i\}$, subject to relations:
\begin{enumerate}
\item $\tau_{n,i}^2=0, 0\leq i\leq n-1;$
\item $\tau_{n;s}\tau_{n;i}=\sigma_{n;s,i}\xi_{n,0},0\leq s<i\leq n-1;$
\item $\tau_{n;s}\tau_{n;i}\tau_{n;j}=\tau_{n;s}\sigma_{n;i,j}\xi_{n,0},0\leq s<i<j\leq n-1;$
\item $\tau_{n;s}\tau_{n;i}\tau_{n;j}\tau_{n;k}=\sigma_{n;s,i}\sigma_{n;j,k}\xi^2_{n,0},0\leq s<i<j<k\leq n-1.$
\end{enumerate}
\end{thm}

To related the Dyer-lashof algebra and invariant theory, let $\mathscr{B}[n]$ be the subalgebra of $(H^*B\V_n)^{GL_n}$ generated by the following elements:
\begin{enumerate}
	\item[(1)] $q_{n,i}$ for $0\leq i\leq n-1$,
	\item[(2)] $R_{n;s}$ for $0\leq s\leq n-1$,
	\item[(3)] $R_{n;s,t}$ for $0\leq s<t\leq n-1$.
\end{enumerate}

In \cite{Mui75}, M\`ui shows that the algebra $\mathscr{B}[n]$ is the image of the restriction in cohomology from the symmetric group $\Sigma_{p^n}$ to $\V_n$. The connection between $\mathscr{B}[n]$ and the dual of the Dyer-Lashof algebra was first observed by Madsen \cite{Mad75} for $p=2$ and by May \cite{Coh-Lad-May76} for $p$ odd. The most precise form was determined by Kechagias.

\begin{thm}[(Kechagias \cite{Kechagias1994}, \cite{Kechagias2004})] \label{thm:dual} As algebras over the Steenrod algebra, $R[n]^*$ is isomorphic to  $\mathscr{B}[n]$ via the isomorphism $\Phi$ given by $\Phi(\xi_{n,i})=-q_{n,i}$, $\Phi(\tau_{n;i})=R_{n;i}$, $0\leq i\leq n-1$ and $\Phi(\sigma_{n;s,i})=R_{n;s,i},0\leq s<i\leq n-1$.
\end{thm}
 
Thus, $R[n]$ is isomorphic to the image of $H_*B\V_n$ in $H_*B\Sigma_{p^n}$  induced by the canonical inclusion. According to Priddy \cite{Priddy1973c}, $R[n]$ is also isomorphic to the image of $H_*B\V_n$ in $H_*QS^0$. 

 When $p=2$, the situation is considerably simpler and is well-understood. The mod 2 Dyer-Lashof algebra $R$ is generated by the symbols $Q^i, i\geq0,$. Therefore, Madsen's description \cite{Mad75} of the dual of $R[n]$ actually follows from May's computation \cite{Coh-Lad-May76} by killing the Bockstein and halving the degree of $Q^i$.

%%%%%%%%%%%%%%%%%%%%%%%%%
%%%%%%%%%%%%%%%%%%%%%%%%%
\section{Additive base of modular (co)invariants}\label{sec:Basis}
In this section, we construct a new basis for $\mathscr{B}[n]^*$, which will be crucial for application in Section \ref{sec: Hopf ring}. Note that via an explicit isomorphism in Theorem \ref{thm:dual}, we also have a basis for $R[n]$. Our method also yields new additive bases of the Dickson-M\`ui invariants $(H^*B\V_n)^{GL_n}$, the Dickson-M\`ui coinvariants $(H_*B\V_n)_{GL_n}$ as well as the cokernel of the restriction map $H^*B\Sigma_{p^n}\rightarrow (H^*B\V_n)^{GL_n}$, which may be useful for other applications.

Let us first introduce an order on the set of tuples $I=(\epsilon_1,i_1,\dots,\epsilon_n, i_n)$ as follows
\begin{enumerate}
\item $(\epsilon_1,i_1)<(\omega_1,j_1)$ if $\epsilon_1+i_1<\omega_1+j_1$ or $i_1+\epsilon_1=j_1+\omega_1$, $\epsilon_1<\omega_1$;
\item $(\epsilon_1,i_1,\dots,\epsilon_k,i_k)<(\omega_1,j_1,\dots,\omega_k,j_k)$ if:\\
(a) $I=(\epsilon_1,i_1,\dots,\epsilon_{k-1},i_{k-1})<(\omega_1,j_1,\dots,\omega_{k-1},j_{k-1})=J$ or\\
(b) $I=J$, $i_k+p^{k-1}\epsilon_k<j_k+p^{k-1}\omega_k$ or\\
(c) $I=J$, $i_k+p^{k-1}\epsilon_k=j_k+p^{k-1}\omega_k$ and $\epsilon_k<\omega_k$.
\end{enumerate}

If $\epsilon_k=\omega_k=0$ for all $k$, this ordering coincides with the lexicographic ordering from the left.

A monomial $q^I=R^{\epsilon_1}_{n;0}q_{n,0}^{i_1}\cdots R^{\epsilon_n}_{n;n-1}q_{n,n-1}^{i_n}$ (or $V^I,x^I$) is  less than $q^J$ (or $V^J,x^J$) if $I<J$. For any homogeneous polynomial $f(e_1,x_1,\dots,e_n,x_n)\in H^*B\V_n$, we denote $LT(f(e_1,x_1,\dots,e_n,x_n))$ the  smallest monomial occurring in $f$, and call it the leading term of $f$.

The next two lemmas compute the leading term of a generator of the Dickson-M\`ui algebra. 

\begin{lem}\label{lm:expand Dickson-Mui}
For $i_s\geq0$, we have
\begin{align*}
q_{n,0}^{i_1}\cdots q_{n,n-1}^{i_n}&=x_1^{(p-1)i_1}x_{2}^{p(p-1)(i_1+i_2)}\cdots x_n^{p^{n-1}(p-1)(i_1+\cdots+i_n)}\\
&\quad+\text{greater monomials of $x_k$'s}.
\end{align*}
\end{lem}
\begin{proof}
For $0\leq s\leq n-1$, the inductive formula
\[
q_{n,s}=q_{n-1,s-1}^p+q_{n-1,s}V_n^{p-1},
\]
allows us to express the Dickson generators $q_{n,s}$ in term of $V_i$'s:
\[
q_{n,s}=(V_s\cdots V_n)^{p-1}+\text{greater monomials of $V_k$'s}.
\]

It follows that
\begin{align*}
q_{n,0}^{i_1}\cdots q_{n,n-1}^{i_n}&=V_1^{(p-1)i_1}\cdots V_n^{(p-1)(i_1+\cdots+i_n)}+\sum_{k}V^{I_k},
\end{align*}
where $I_k>((p-1)i_1,\dots,(p-1)(i_1+\cdots+i_n))$.

Moreover, by the definition
\[
V_s=\prod_{\lambda_i\in\fp}(\lambda_1x_1+\cdots+\lambda_{n-1}x_{n-1}+x_n)=x_n^{p^{n-1}}+\text{greater monomials of $x_k$'s},
\]
it is easy to verify that if $I<J$ then $LT(V^I)<LT(V^J)$.

So that, we have
\begin{align*}
q_{n,0}^{i_1}\cdots q_{n,n-1}^{i_n}&=x_1^{(p-1)i_1}x_{2}^{p(p-1)(i_1+i_2)}\cdots x_n^{p^{n-1}(p-1)(i_1+\cdots+i_n)}\\
&\quad+\text{greater monomials of $x_k$'s}.
\end{align*}

The proof is complete.
\end{proof}

 For any string of integers $I=(\epsilon_1,i_1,\dots,\epsilon_n,i_n),$ with  $i_1\in\mathbb{Z},i_s\geq0, 2\leq s\leq n$, and $\epsilon_s\in \{0,1\}$, we put $b(I)=\sum_s\epsilon_s$ and $m(I)=\max\{\epsilon_s:1\leq s\leq n\}$. 
\begin{lem}\label{lm:expand basis of B}
For $I=(\epsilon_1,i_1,\dots,\epsilon_n,i_n),$ with $i_1\in\mathbb{Z}, i_s\geq0, 2\leq s\leq n ,\epsilon_s\in \{0,1\}$, and $i_1-m(I)+b(I)\geq0$, we have
\begin{align*}
R^{\epsilon_1}_{n;0}&q_{n,0}^{i_1}\cdots R^{\epsilon_n}_{n;n-1}q_{n,n-1}^{i_n}\\
=&(-1)^{\epsilon_2+\cdots+(n-1)\epsilon_n}e_1^{\epsilon_1}x_1^{(p-1)(i_1+b(I))-\epsilon_1}\cdots e^{\epsilon_n}_nx_n^{p^{n-1}(p-1)(i_1+\cdots+i_n+b(I))-p^{n-1}\epsilon_n}\\
&\quad+\text{greater monomials of $e_k$'s and $x_k$'s}.
\end{align*}
\end{lem}
\begin{proof}
From the proof of Lemma \ref{lm:expand Dickson-Mui}, we have
\begin{align*}
q_{n,0}^{i_1}\cdots q_{n,n-1}^{i_n}&=V_1^{(p-1)i_1}\cdots V_n^{(p-1)(i_1+\cdots+i_n)}+\sum_kV^{I_k}\\
&=L_1^{(p-1)i_1}\frac{L_2^{(p-1)(i_1+i_2)}}{L_1^{(p-1)(i_1+i_2)}}\cdots\frac{L_n^{(p-1)(i_1+\cdots+i_n)}}{L_{n-1}^{(p-1)(i_1+\cdots+i_n)}}+\sum_kV^{I_k}\\
&=\frac{L_n^{(p-1)(i_1+\cdots+i_n)}}{L_1^{(p-1)i_2}\cdots L_{n-1}^{(p-1)i_n}}+\sum_kV^{I_k},
\end{align*}
where $I_k>((p-1)i_1,\dots,(p-1)(i_1+\cdots+i_n))$.

Since $R_{n;s}=M_{n;s}L_n^{p-2}$, for $0\leq s\leq n-1$, we obtain
{\allowdisplaybreaks
\begin{align*}
R_{n;0}^{\epsilon_1}&q_{n,0}^{i_1}\cdots R_{n;n-1}^{\epsilon_n}q_{n,n-1}^{i_n}\\
&=M_{n;0}^{\epsilon_1}\cdots M_{n;n-1}^{\epsilon_n}\frac{L_n^{(p-1)(i_1+\cdots+i_n+b(I))-b(I)}}{L_1^{(p-1)i_2}\cdots L_{n-1}^{(p-1)i_n}}+\sum_kM_{n;0}^{\epsilon_1}\cdots M_{n;n-1}^{\epsilon_n}L_n^{(p-2)b(I)}V^{I_k}\\
&=M_{n;0}^{\epsilon_1}\cdots M_{n;n-1}^{\epsilon_n}V_1^{(p-1)(i_1+b(I))-b(I)}\cdots V_n^{(p-1)(i_1+\cdots+i_n+b(I))-b(I)}\\
&\quad+\sum_kM_{n;0}^{\epsilon_1}\cdots M_{n;n-1}^{\epsilon_n}V^{I'_k},
\end{align*}
}
where $I'_k>((p-1)(i_1+b(I))-b(I),\dots,(p-1)(i_1+\cdots+i_n+b(I))-b(I))$.

By assumption, $i_s\geq0, 2\leq s \leq n$ and $i_1-m(I)+b(I)\geq0$, we can proceed as in the proof of Lemma \ref{lm:expand Dickson-Mui}, to obtain
\begin{align*}
R_{n;0}^{\epsilon_1}&q_{n,0}^{i_1}\cdots R_{n;n-1}^{\epsilon_n}q_{n,n-1}^{i_n}\\
&=M_{n;0}^{\epsilon_1}\cdots M_{n;n-1}^{\epsilon_n}x_1^{(p-1)(i_1+b(I))-b(I)}\cdots x_n^{p^{n-1}(p-1)(i_1+\cdots+i_n+b(I))-p^{n-1}b(I)}\\
&\quad\quad+\sum_kM_{n;0}^{\epsilon_1}\cdots M_{n;n-1}^{\epsilon_n}x^{J_k},
\end{align*}
where $J_k>((p-1)(i_1+b(I))-b(I),\dots,p^{n-1}(p-1)(i_1+\cdots+i_n+b(I))-p^{n-1}b(I))$.

Moreover, for $0\leq s\leq n-1$,
\begin{align*}
M_{n;s}&=(-1)^sx_1x_2^p\cdots x_{s}^{p^{s-1}}e_{s+1}x_{s+2}^{p^{s+1}}\cdots x_{n}^{p^{n-1}}\\
&\quad+\text{ greater monomials of $e_k$'s and $x_k$'s},
\end{align*}
in other words, $x_1x_2^p\cdots x_{s}^{p^{s-1}}e_{s+1}x_{s+2}^{p^{s+1}}\cdots x_{n}^{p^{n-1}}$ is the least monomial occurring non-trivially in $M_{n;s}$. Indeed, it is sufficient to compare the order of $n$ following monomials.
\begin{align*}
e_1x_2x_3^p\cdots x_{s}^{p^{s-2}}x_{s+1}^{p^{s-1}}x_{s+2}^{p^{s+1}}\cdots x_n^{p^{n-1}},\\
\cdots\cdots\cdots\cdots\cdots\cdots\cdots\cdots\cdots\cdots\cdots\cdots\\
x_1x_2^{p}x_3^{p^2}\cdots x_{s}^{p^{s-1}}e_{s+1}x_{s+2}^{p^{s+1}}\cdots x_n^{p^{n-1}},\\
\cdots\cdots\cdots\cdots\cdots\cdots\cdots\cdots\cdots\cdots\cdots\cdots\\
x_1x_2^px_3^{p^2}\cdots x_{s}^{p^{s-1}}x_{s+1}^{p^{s+1}}\cdots x_{n-1}^{p^{n-1}}e_n.
\end{align*}

and this can be verified directly.

In addition, it is clear that if $I<J$, $I'<J'$ and both $I',J'$ of the form $(0,t_1,\dots,0,t_n)$, then $x^Ix^{I'}<x^Jx^{J'}$.
Therefore, 
\begin{align*}
LT(M_{n;0}^{\epsilon_1}&\cdots M_{n;n-1}^{\epsilon_n}x_1^{(p-1)(i_1+b(I))-b(I)}\cdots x_n^{p^{n-1}(p-1)(i_1+\cdots+i_n+b(I))-p^{n-1}b(I)})\\
&=e_1^{\epsilon_1}x_1^{(p-1)(i_1+b(I))-\epsilon_1}\cdots e^{\epsilon_n}_nx_n^{p^{n-1}(p-1)(i_1+\cdots+i_n+b(I))-p^{n-1}\epsilon_n}.
\end{align*}

Furthermore, it is easy to verify that the leading term of $M_{n;0}^{\epsilon_1}\cdots M_{n;n-1}^{\epsilon_n}x^{J_k}$ is greater than $e_1^{\epsilon_1}x_1^{(p-1)(i_1+b(I))-\epsilon_1}\cdots e^{\epsilon_n}_nx_n^{p^{n-1}(p-1)(i_1+\cdots+i_n+b(I))-p^{n-1}\epsilon_n}$.

Combining these facts, we have the assertion of the lemma.
\end{proof}

\begin{pro}\label{pro:basis of Dickson-Mui algebra}
For any $n\geq1$, $(H^*B\V_n)^{GL_n}$, considered as an $\fp$-vector space, has a basis consisting of all elements $q^I=R^{\epsilon_1}_{n;0}q_{n,0}^{i_1}\cdots R^{\epsilon_n}_{n;n-1}q_{n,n-1}^{i_n}$ for $i_1\in\mathbb{Z}, i_s\geq0, 2\leq s\leq n ,\epsilon_s\in \{0,1\}$ and $i_1-m(I)+b(I)\geq0.$
\end{pro}
\begin{proof}
Theorem \ref{thm:invariant of G} implies that the set $\{q^I=R^{\epsilon_1}_{n;0}q_{n,0}^{i_1}\cdots R^{\epsilon_n}_{n;n-1}q_{n,n-1}^{i_n}:i_1-m(I)+b(I)\geq0\}$ spans $(H^*B\V_n)^{GL_n}$. By Lemma \ref{lm:expand basis of B}, this set is linearly independent.
\end{proof}

Keeping the above notation, we have the following result for $\mathscr{B}[n]$.

\begin{pro}\label{pro:basis of B(p)} For any $n\geq 1$,
the set of elements $q^I=R^{\epsilon_1}_{n;0}q_{n,0}^{i_1}\cdots R^{\epsilon_n}_{n;n-1}q_{n,n-1}^{i_n}$ for $i_1\in\mathbb{Z}, i_s\geq0, 2\leq s\leq n ,\epsilon_s\in \{0,1\}$ and $2i_1+b(I)\geq0,$
forms an additive basis for $\mathscr{B}[n]$. 
\end{pro}
\begin{proof}
It is clear from Proposition \ref{pro:basis of Dickson-Mui algebra} that this set is linearly independent. It spans $\mathscr{B}[n]$ because $R_{n;s,t}=R_{n;s}R_{n;t}q_{n,0}^{-1}$ for all pair $0\leq s<t\leq n-1$.
\end{proof}
\begin{cor}
For any $n\geq 1$, as an $\fp$-vector space, the cokernel of the restriction map $H^*B\Sigma_{p^n}\to (H^*B\V_n)^{GL_n}$ has a basis consisting of all elements $[R^{\epsilon_1}_{n;0}q_{n,0}^{i_1}\cdots R^{\epsilon_n}_{n;n-1}q_{n,n-1}^{i_n}]$ for $i_1\in\mathbb{Z}, i_s\geq0, 2\leq s\leq n ,\epsilon_s\in \{0,1\}$ and $m(I)-b(I)\leq i_1<-b(I)/2$.
\end{cor}

For $k\geq0$,  the subspace of $\mathscr{B}[n]$ generated by $\{R^{\epsilon_1}_{n;0}q_{n,0}^{i_1}\cdots R^{\epsilon_n}_{n;n-1}q_{n,n-1}^{i_n}:2i_1+b(I)\geq k\}$ is a subalgebra of $\mathscr{B}[n]$, which is denoted by $\mathscr{B}_k[n]$. It is immediate that $\mathscr{B}_0[n]=\mathscr{B}[n]$.

We now turn to homology. Let $u_i\in H_1B\V_n$ be the dual of $e_i$ and let $v_i\in H_{2}B\V_n$ be the dual of $x_i$. Then $H_*B\V_n$ is the tensor product of the exterior algebra generated by $u_i$'s and the divided power algebra generated by $v_i$'s. We denote by $v_i^{[t]}$ the $t$-th divided power of $v_i$.  Since $R[n]$ is isomorphic to $\mathscr{B}[n]^*$, $R[n]$ is considered the quotient algebra of $(H_*B\V_n)_{GL_n}$. The following theorem provides an additive basis for $\mathscr{B}[n]^*$ and hence for $R[n]$.
\begin{thm}[({Compare with \cite[Theorem 3.10]{Tur97}})]\label{thm:basis in coinvariant}
For $k\geq0$, the set of all elements 
\[
[u_1^{\epsilon_1}v_1^{[(p-1)(i_1+b(I))-\epsilon_1]}\cdots u_n^{\epsilon_n}v_n^{[p^{n-1}(p-1)(i_1+\cdots+i_n+b(I))-p^{n-1}\epsilon_n]}],
\]
for $i_1\in\mathbb{Z}, i_s\geq0, 2\leq s\leq n ,\epsilon_s\in \{0,1\}$, and $2i_1+b(I)\geq k$ provides an additive basis for $\mathscr{B}_k[n]^*$, the dual of $\mathscr{B}_k[n]$.
\end{thm}
\begin{proof}
Denote 
\begin{align*}
\mathcal{Q}(\epsilon_1,i_1,\dots,\epsilon_n,i_n)&=\\
(-1)^{\epsilon_2+\cdots+(n-1)\epsilon_n}&u_1^{\epsilon_1}v_1^{[(p-1)(i_1+b(I))-\epsilon_1]}\cdots u^{\epsilon_n}_nv_n^{[p^{n-1}(p-1)(i_1+\cdots+i_n+b(I))-p^{n-1}\epsilon_n]}.
\end{align*}

From Lemma \ref{lm:expand basis of B}, we see that
\begin{equation*}
\begin{split}
\left<R_{n;0}^{\epsilon_1}q_{n,0}^{i_1}\cdots R_{n;n-1}^{\epsilon_n}q_{n,n-1}^{i_n},\mathcal{Q}(\omega_1,s_1,\dots,\omega_n,s_n)\right>\quad\quad \quad &\\
=
\left\{\begin{array}{cc}
0,& (\omega_1,s_1,\dots,\omega_n,s_n)<(\epsilon_1,i_1,\dots,\epsilon_n,i_n);\\
1,&(\omega_1,s_1,\dots,\omega_n,s_n)=(\epsilon_1,i_1,\dots,\epsilon_n,i_n).
\end{array}\right.&
\end{split}
\end{equation*}
Therefore, the set of all $[\mathcal{Q}(\epsilon_1,i_1,\dots,\epsilon_n,i_n)]$ satisfying the condition in the theorem provides a basis of $\mathscr{B}_k[n]^*$.

Moreover, since $\mathscr{B}_k[n]^*$ is a quotient algebra of $(H_*B\V_n)_{GL_n}$, 
\begin{align*}
&[\mathcal{Q}(\epsilon_1,i_1,\dots,\epsilon_n,i_n)]\\
&\quad=[u_1^{\epsilon_1}v_1^{[(p-1)(i_1+b(I))-\epsilon_1]}\cdots u_n^{\epsilon_n}v_n^{[p^{n-1}(p-1)(i_1+\cdots+i_n+b(I))-p^{n-1}\epsilon_n]}].
\end{align*}
Hence, we have the assertion of the theorem.
\end{proof}

Note that when $k=0$, the basis in Theorem \ref{thm:basis in coinvariant} is not the dual basis discussed in Proposition \ref{pro:basis of B(p)}. For $p=2$, Turner \cite{Tur97} constructed an analogous basis of $\mathscr{B}[n]^*$, which is considered a basis of $R[n]$. In \cite{Cho2011}, we work at $p=2$ and constructed a new basis for $R[n]$ and pointed out the relationship between these bases. We believe that the method used can be employed to construct an analogous basis for $R[n]$ for $p$ odd.

The following proposition is similar to Theorem \ref{thm:basis in coinvariant}. The proof is thus left to the interested reader. 

\begin{pro}[({Compare with \cite[Theorem 3.10]{Tur97}})]\label{pro:basis of H_GL}
For $n\geq 1$, the set of all elements
\[
[u_1^{\epsilon_1}v_1^{[(p-1)(i_1+b(I))-\epsilon_1]}\cdots u_n^{\epsilon_n}v_n^{[p^{n-1}(p-1)(i_1+\cdots+i_n+b(I))-p^{n-1}\epsilon_n]}],
\]
for $i_1\in\mathbb{Z}, i_s\geq0, 2\leq s\leq n ,\epsilon_s\in \{0,1\}$, and $i_1+b(I)-m(I)\geq0$ provides an additive basis for $(H_*B\V_n)_{GL_n}$.
\end{pro}

Let $I_k$ be the ideal of $R$ generated by all monomials of excess less than $k$ and let $R_k$ denote the quotient algebra $R/I_k$. For $k\geq 0$, let $R_k[n]$ be the subspace of $R_k$ spanned by all monomial of length $n$. 
Madsen \cite{Mad75} and May \cite{Coh-Lad-May76} show that $R[n]^*=R_0[n]^*\cong \mathscr{B}_0[n]$. The following provides a more explicit statement.

\begin{pro}\label{pro:R_k and B_k[n]}
As algebras, $R_k[n]^*\cong \mathscr{B}_k[n]$ via the isomorphism given in Theorem \ref{thm:dual}.
\end{pro}
\begin{proof}
For a string of integers $e=(e_1,\dots, e_j)$ such that $1\leq e_1<\cdots<e_j\leq n$, we put 
\[
L_{n;e}=\left\{
\begin{array}{ll}
K_{n;e_1,e_2}+\cdots+K_{n;e_{j-1},e_j},&\text{if $j$ is even},\\
K_{n;e_1,e_2}+\cdots+K_{n;e_{j-2},e_{j-1}}+J_{n;e_j},&\text{if $j$ is odd,}
\end{array}
\right.
\]
and $L_{n;e}$ is the string of all zeros if $e$ is empty.
Here we mean $(\epsilon_1,i_1,\dots,\epsilon_n,i_n)+(\epsilon'_1,j_1,\dots,\epsilon'_n,j_n)$ to be the string $(\omega_1,t_1,\dots,\omega_n,t_n)$ with $t_s=i_s+j_s$ and $\omega_s=\epsilon_s+\epsilon'_s\ (\text{mod }2)$.

In \cite[p.38]{Coh-Lad-May76}, May shows that any string $I$ of non-negative excess can be uniquely expressed in the form
\[
I=\sum_{i=0}^{n-1}t_iI_{n,i}+L_{n;e},
\]
for some string $e$, and $exc(I)=2t_0+exc(L_{n;e})$. By the same argument of the proof of Theorem 3.7 in \cite[p.29]{Coh-Lad-May76}, we obtain that the set of all monomials $$\xi_{n,0}^{i_1}\cdots \xi_{n,n-1}^{i_n}(\sigma_{n;e_1,e_2}\cdots\sigma_{n;e_{j-2},e_{j-1}})^{\epsilon_1}\tau_{n;e_j}^{\epsilon_2},\quad 2i_1+\epsilon_2\geq k$$
provides an additive basis of $R_k[n]^*$.

From relation (ii) in Theorem \ref{thm: relation of DL},  we see that the above monomials can be written in the form (up to a sign)
\[
\tau_{n;0}^{\epsilon_1}\xi_{n,0}^{i_1}\cdots \tau_{n;n-1}^{\epsilon_n}\xi_{n,n-1}^{i_n}, \quad 2i_1+b(I)\geq k.
\]

It implies that the set of all monomials $\tau_{n;0}^{\epsilon_1}\xi_{n,0}^{i_1}\cdots \tau_{n;n-1}^{\epsilon_n}\xi_{n,n-1}^{i_n}, 2i_1+b(I)\geq k$ is a basis of $R_k[n]^*$.

By the definition of $\mathscr{B}_k[n]$ and Theorem \ref{thm:dual} we have the assertion of the proposition.
\end{proof}

\section{The Hopf ring structure of $\{H_*QS^k\}_{k\geq 0}$}\label{sec: Hopf ring}

In this section, we use results of the previous sections to give a complete set of relations for $\{H_*QS^k\}_{k\geq0}$ as a Hopf ring.

Let $[1]\in H_*QS^0$ be the image of non-base point generator of $H_0S^0$ under the map induced by the canonical inclusion $S^0\hookrightarrow QS^0$ and let $\sigma\in H_*QS^1$ be the image of the generator of $H_1S^1$ under the homomorphism induced by the inclusion $S^1\hookrightarrow QS^1$. Note that the element $\sigma$ is usually known as the homology suspension element because $\sigma\circ x$ is the homology suspension of $x$. Write $P[X]$ for the free graded communicative algebra over $\fp$ generated by the set $X$. We first recall the following fundamental calculation by Dyer-Lashof \cite{Dye-Las62} and May \cite{Coh-Lad-May76}.
\begin{thm}[(Dyer-Lashof \cite{Dye-Las62}, May \cite{Coh-Lad-May76})]
The mod $p$ homology of $\{QS^k\}_{k\geq0}$ is given by
\begin{align*}
H_*QS^0&=P[Q^I[1]:I \text{ admissible}, exc(I)+\epsilon_1> 0]\otimes \fp[\mathbb{Z}],\\
H_*QS^k&=P[Q^I(\sigma^{\circ k}):I \text{ admissible}, exc(I)+\epsilon_1> k], k>0.
\end{align*}
\end{thm}

The interaction between the action of the Steenrod algebra, the Dyer-Lashof algebra and
the circle product were worked out in \cite{Coh-Lad-May76} and \cite{May1971}.
\begin{thm}[(May \cite{Coh-Lad-May76}, \cite{May1971})]\label{thm:action of DL on HQSn}
For $b,f\in H_*QS^k$,
\begin{enumerate}
\item $P^s_*(b\circ f)=\sum_{i} P^i_*(b)\circ P_*^{s-i}(f)$ and $\beta(b\circ f)=\beta(b)\circ f+(-1)^{{\rm deg}b}b\circ\beta(f)$.
\item $Q^s(b)\circ f=\sum_iQ^{s+i}(b\circ P^i_*(f))$.
\item $\beta Q^s(b)\circ f=\sum_i\beta Q^{s+i}(b\circ P^i_*(f))-\sum_i(-1)^{{\rm deg}b}Q^{s+i}(b\circ P^i_*\beta(f))$.
\end{enumerate}
\end{thm}

In \cite{Kahn.Priddy1978a}, Kahn and Priddy constructed the transfer
\[
tr^{(1)} \colon (B\V_1)_+\rightarrow QS^0,
\]
with the property that the induced map in homology $tr^{(1)}_* \colon H_*(B\V_1)_+\rightarrow H_*QS^0$ sends $u^{\epsilon}v^{[i(p-1)-\epsilon]}$ to $(-1)^i\beta^{\epsilon}Q^i[1]$ and others to zero.  Higher transfer can also be defined.  Let $\psi \colon \Sigma_m\times\Sigma_n\rightarrow\Sigma_{mn}$ be the permutation product of symmetric groups \cite[p.50]{Madsen.Milgram1979}. The composition  
\[
I_n \colon \V_n=\V_1\times\cdots\times \V_1\hookrightarrow \Sigma_p\times\cdots\times\Sigma_p\xrightarrow{\psi\times\cdots\times\psi}\Sigma_{p^n},
\]
is the canonical inclusion, embedding $\V_n$ into $\Sigma_{p^n}$ as the unique elementary abelian $p$-group that acts transitively on the set $\V_n$ itself.  Madsen and Milgram \cite[Theorem 3.10]{Madsen.Milgram1979} proved that there exists a commutative diagram
\[
\bfig
\square(100,0)<1000,500>[B\V_n`B\Sigma_{p^n}`QS^0\times\cdots\times QS^0`QS^0;BI_n`tr^{(1)}\times\cdots\times tr^{(1)}`i`\mu]
\efig
\]
where $\mu$ is the composition product in $QS^0$. This gives rise to the transfer
\[
tr^{(n)}=\mu\circ (tr^{(1)}\times\cdots\times tr^{(1)}):B\V_n\rightarrow QS^0,
\]
so that the induced transfer in homology $tr^{(n)}_*:H_*B\V_n\rightarrow H_*QS^0$ sends the ``external product'' in $H_*B\V_n$ (with respect to the decomposition $B\V_n\simeq B\V_r\times B\V_{n-r}$) to the circle product in $H_*QS^0$. In other words, we have
\begin{align*}
tr^{(n)}_*(u_1^{\epsilon_1}&v_1^{[i_1(p-1)-\epsilon_1]}\cdots u_n^{\epsilon_n}v_n^{[i_n(p-1)-\epsilon_n]})\\
&=tr^{(1)}_*(u_1^{\epsilon_1}v_1^{[i_1(p-1)-\epsilon_1]})\circ\cdots\circ tr^{(1)}_*(u_n^{\epsilon_n}v_n^{[i_n(p-1)-\epsilon_n]}).
\end{align*}

Since $tr^{(n)}=i\circ BI_n$ and $GL_n$ is the ``Weyl group'' of the inclusion $\V_n\subset \Sigma_{p^n}$, we have an important feature of the map $tr^{(n)}_*$ is that they factor through the coinvariant of the general linear group. In other words, there exists a map $(H_*B\V_n)_{GL_n}\to H_*QS^0$ such that the diagram
\[
\bfig
\Vtriangle/>`>`<-/[H_*B\V_n`H_*QS^0`{(H_*B\V_n)}_{GL_n};tr^{(n)}_*`p`]
\efig
\]
is commutative.

Moreover, we have the following.
\begin{pro}\label{pro:factor of transfer}
The transfer $tr^{(n)}_*$ factors through $\mathscr{B}[n]^*$. In other words, there exists a map $\mathscr{B}[n]^*\to H_*QS^0$ such that the diagram
\[
\xymatrix{
H_*B\V_n\ar[rr]^{tr_*^{(n)}}\ar[dr]_{p}\ar[ddr]_{p'}&&H_*QS^0\\
&(H_*B\V_n)_{GL_n}\ar[ur]\ar[d]^{p''}\\
&\mathscr{B}[n]^*\ar[uur]
}
\]
commutes, where $p,p'$ and $p''$ are the canonical projections.
\end{pro}
\begin{proof}
Since $tr_{(n)}^*=(BI_n)^*\circ i^*$, the image of $tr_{(n)}^*$ is contained in the image of the restriction $(BI_n)^*:H^*B\Sigma_{p^n}\rightarrow H^*B\V_n$. Moreover, from M\`ui \cite[Chapter 2, Theorem 6.1]{Mui75}, the image of the restriction $(BI_n)^*$ is $\mathscr{B}[n]\subset (H^*B\V_n)^{GL_n}$. Therefore, the assertion of the proposition is immediate after taking dual.
\end{proof}

For any $I=(\epsilon_1,i_1,\dots,\epsilon_n,i_n)$, with $i_1\in\mathbb{Z}$, $i_s\geq0, 2\leq s\leq n$, $\epsilon_s\in \{0,1\}$, and $i_1+b(I)-m(I)\geq 0$, let $E_{(\epsilon_1, i_1,\dots,\epsilon_n, i_n)}$ is the dual of $R^{\epsilon_1}_{n;0}q_{n,0}^{i_1}\cdots R^{\epsilon_n}_{n;n-1}q_{n,n-1}^{i_n}$ with respect to the monomials basis given in Proposition \ref{pro:basis of Dickson-Mui algebra}; and we use the same notation $E_{(\epsilon_1, i_1,\dots,\epsilon_n, i_n)}$ to denote its image under the transfer $tr^{(n)}_*$.
In particular, $E_{(\epsilon,k)}=(-1)^k\beta^\epsilon Q^k[1]$.

We have another description of the homology of $\{QS^k\}_{k\geq0}$ as follows.
\begin{thm}[({Compare with \cite[Theorem 2.7]{Eccles1997}})]\label{thm:describe though dickson coninvariant}
The mod $p$ homology of $QS^k$ is given by
\begin{align*}
H_*QS^0=P[E_{(\epsilon_1,i_1+b(I))}\circ\cdots\circ &E_{(\epsilon_n,p^{n-1}(i_1+\cdots+i_n+b(I))-\Delta_n\epsilon_n)}:\\
& \quad\quad\quad\quad n\geq 1, 2i_1+b(I)+\epsilon_1>0]\otimes \fp[\mathbb{Z}],
\end{align*}
and for $k>0$,
\begin{align*}
H_*QS^k=P[\sigma^{\circ k}\circ E_{(\epsilon_1,i_1+b(I))}\circ\cdots\circ &E_{(\epsilon_n,p^{n-1}(i_1+\cdots+i_n+b(I))-\Delta_n\epsilon_n)}:\\
& \quad\quad\quad\quad n\geq 1, 2i_1+b(I)+\epsilon_1>k],
\end{align*}
where $\Delta_s=\frac{p^{s-1}-1}{p-1}$.
\end{thm}

For the proof of this theorem, we need a couple of lemmas.
\begin{lem}\label{lm:relation E and Q}
For $n\geq1$,
\begin{multline*}
%\begin{align*}
\sigma^{\circ k}\circ E_{(\epsilon_1,i_1+b(I))}\circ\cdots \circ E_{(\epsilon_s,p^{s-1}(i_1+\cdots+i_s+b(I))-\Delta_ s\epsilon_s)} \cdots \circ    E_{(\epsilon_n,p^{n-1}(i_1+\cdots+i_n+b(I))-\Delta_n\epsilon_n)}\\
=(-1)^{ni_1+\cdots+i_n+nb(I)} \beta^{\epsilon_1}Q^{j_1}\cdots\beta^{\epsilon_n}Q^{j_n}(\sigma^{\circ k})+\sum \alpha_KQ^K(\sigma^{\circ k}),
%\end{align*}
\end{multline*}
where
\begin{align*}
j_s&=p^{n-s}(i_1+\cdots+i_s+b(I))+\sum_{\ell=0}^{n-s-1}p^{\ell}(p^{n-s-\ell}-1)i_{s+\ell+1}-\delta_n(s),
\end{align*}
$\delta_n(s)=p^{n-s-1}\epsilon_n+\cdots+\epsilon_{s+1}$ if $s< n$ and $\delta_n (n) = 0$; $\alpha_K\in\fp$ and $exc(K)< exc(J)=2i_1+b(I)$. 
\end{lem}

\begin{proof}
The $n=1$ case is immediate.
Using Theorem \ref{thm:action of DL on HQSn} and Nishida's relations, we obtain the assertion of the lemma for $n=2$.

We shall prove the case $n\geq 3$ by the induction.
It is easy to see that
\begin{align*}
x&=E_{(\epsilon_2,p(i_1+i_2+b(I))-\epsilon_2)}\circ\cdots\circ E_{(\epsilon_n,p^{n-1}(i_1+\cdots+i_n+b(I))-\Delta_n\epsilon_n)}\\
&=E_{(\epsilon_2,i_2'+b(I'))}\circ\cdots\circ E_{(\epsilon_n,p^{n-2}(i_2'+\cdots+i_n'+b(I'))-\Delta_{n-1}\epsilon_n)},
\end{align*}
where $I'=(\epsilon_2,i_2',\dots,\epsilon_n,i_n')$ and
\begin{align*}
i_2'&=pi_1+pi_2+p\epsilon_1+(p-1)b(I')-\epsilon_2;\\
i_s'&=pi_s+\epsilon_{s-1}-\epsilon_s, 3\leq s\leq n.
\end{align*}
Applying the inductive hypothesis, one gets
{\allowdisplaybreaks
\begin{align*}
x=(-1)^{(n-1)i_2'+\cdots+i_n'+(n-1)b(I')}\beta^{\epsilon_2}Q^{j'_2}\cdots\beta^{\epsilon_n}Q^{j'_n}[1]+\text{other terms of smaller excess},
\end{align*}
}
where
\begin{align*}
j'_s&=p^{n-s}(i'_2+\cdots+i'_s+b(I'))+\sum_{\ell=0}^{n-s-1}p^{\ell}(p^{n-s-\ell}-1)i'_{s+\ell+1}-\delta_{n-1}(s-1).
\end{align*}

Applying Theorem \ref{thm:action of DL on HQSn} and Nishida's relation, we have
{\allowdisplaybreaks
\begin{align*}
y&=E_{(\epsilon_1,i_1+b(I))}\circ E_{(\epsilon_2,p(i_1+i_2+b(I))-\epsilon_2)}\circ\cdots\circ E_{(\epsilon_n,p^{n-1}(i_1+\cdots+i_n+b(I))-\Delta_n\epsilon_n)}\\
&=(-1)^{(n-1)i_2'+\cdots+i'_n+(n-1)b(I')+i_1+b(I)}\sum_{k_2}\beta^{\epsilon_1}Q^{i_1+b(I)+k_2}(P^{k_2}_*(\beta^{\epsilon_2}Q^{j'_2}\cdots\beta^{\epsilon_n}Q^{j'_n}[1]))+\text{others}\\
&=(-1)^{k}\sum_{k_2,\dots,k_n}\binom{(p-1)(j'_2-k_2)-\epsilon_2}{k_2-pk_3}\cdots\binom{(p-1)(j'_n-k_n)-\epsilon_n}{k_n}\\
&\quad \times \beta^{\epsilon_1}Q^{i_1+b(I)+k_2}\beta^{\epsilon_2}Q^{j'_2-k_2+k_3}\cdots\beta^{\epsilon_n}Q^{j'_n-k_n}[1]+\text{others}.
\end{align*}
}
If all the binomial coefficients are nonzero, it is necessary that 
\[
(p-1) j_2' - \epsilon_2 + pk_3 \geq pk_2, \ldots, (p-1)j_n' - \epsilon_n \geq pk_n. 
\]
The excess of $(\epsilon_1,i_1+b(I)+k_2,\epsilon_2,j'_2-k_2+k_3,\dots,\epsilon_n,j'_n-k_n)$ equals 
\[
2 (i_1 + b(I) + k_2) - 2\epsilon_1 - 2(p-1) \left(\sum_{s=2}^n j_s' - k_2\right) + b(I) 
\]
which is greatest when $k_2$ is greatest possible. This happens when 
\begin{align*}
k_n&=(p-1)(i_1+\cdots+i_n+b(I))-\epsilon_n,\\
k_s&=k_{s+1}+(p-1)j_s-\epsilon_s, 2\leq s\leq n-1,
\end{align*}
 in which case we get
\begin{align*}
y=(-1)^k\beta^{\epsilon_1}Q^{j_1}\cdots\beta^{\epsilon_n}Q^{j_n}[1]+\text{other terms of smaller excess}.
\end{align*}
It is easy to verify that $(-1)^k=(-1)^{ni_1+\cdots+i_n+nb(I)}$, $j_s - p j_{s+1} = - i_{s+1} - \epsilon_{s+1}$ for all $s <n$ and $j_n = \sum_{s=1}^n i_s + b(I)$. Hence $exc (J) = 2i_1 + b(I)$.  

Finally, since $\sigma^{\circ k}\circ \beta^\epsilon Q^j[1]=\beta^\epsilon Q^j(\sigma^{\circ k})$, then $\sigma^{\circ k}\circ y$ can be written in the needed form.
\end{proof}
\begin{lem}\label{lm:bijective 2}
The function mapping $I=(\epsilon_1,i_1+b(I),\dots,\epsilon_n,p^{n-1}(i_1+\cdots+i_n+b(I))-\Delta_n\epsilon_n), 2i_1+b(I)>k$, to the admissible string $J=(\epsilon_1, j_1,\dots,\epsilon_n,j_n),$ with $exc(J)>k$, given as in Lemma \ref{lm:relation E and Q}, is a bijection.
\end{lem}
\begin{proof}
By direct inspection, we have for $s \geq 0$, 
\[
p j_{s+1} - j_{s} = i_{s+1} + \epsilon_{s+1} \geq \epsilon_{s+1}. 
\] 
The fact that this correspondence is a bijection is immediate. 
\end{proof}

We are now ready to prove Theorem  \ref{thm:describe though dickson coninvariant}. 
\begin{proof}[of Theorem \ref{thm:describe though dickson coninvariant}]
From Lemma \ref{lm:relation E and Q}, the set of elements $\sigma^{\circ k}\circ E_{(\epsilon_1,i_1+b(I))}\circ\cdots\circ E_{(\epsilon_n,p^{n-1}(i_1+\cdots+i_n+b(I))-\Delta_n\epsilon_n)}$ belongs to the indecomposable quotient (with respect to the star product) $QH_*QS^k$ and it is linearly independent. 

Moreover, the degree of $$\sigma^{\circ k}\circ E_{(\epsilon_1,i_1+b(I))}\circ\cdots\circ E_{(\epsilon_n,p^{n-1}(i_1+\cdots+i_n+b(I))-\Delta_n\epsilon_n)}$$ is equal to the degree of $\beta^{\epsilon_1}Q^{j_1}\cdots\beta^{\epsilon_n}Q^{j_n}(\sigma^{\circ k})$.

Finally, from Lemma \ref{lm:bijective 2}, we see that this set generates $QH_*QS^k$ in each degree.
\end{proof}

Thus, the elements $E_{(\epsilon,s)}=(-1)^s\beta^{\epsilon}Q^{s}[1]$ and $\sigma$ generate $\{H_*QS^k\}_{k\geq0}$ as a Hopf ring. The problem is to find a complete set of relations. It is solved by investigating the structure of the dual of $\mathscr{B}_k[n]$.

Let $E^\epsilon(s)\in H_*QS^0[[s]]$, $\epsilon=0,1$, be defined by
\[
E^0(s)=\sum_{k\geq0}E_{(0,k)}s^k,\quad E^1(s)=\sum_{k\geq 1}E_{(1,k)}s^k.
\]

Since the coproduct on the $E_{(\epsilon_k,k)}$ arises from the coproduct on $[u^{\epsilon_k}v^{[k(p-1)-\epsilon_k]}]$ in $H_{2k(p-1)-\epsilon_k}(B\V_1)_{GL_1}$,
\begin{align*}
\psi(E_{(0,k)})&=\sum_{i+k=j}E_{(0,i)}\otimes E_{(0,j)},\\
\psi(E_{(1,k+1)})&=\sum_{i+j=k}(E_{(0,i)}\otimes E_{(1,j+1)}+E_{(1,i+1)}\otimes E_{(0,j)}).
\end{align*}
Therefore,
\begin{align*}
\psi(E^0(s))&=E^0(s)\otimes E^0(s);\\
\psi(E^1(s))&=E^0(s)\otimes E^1(s)+E^1(s)\otimes E^0(s).
\end{align*}
For $x\in H_*QS^k$ we define $Q^0(s)x, Q^1(s)x\in H_{*}QS^k[[s]]$ as follows
\[
Q^0(s)x=\sum_{k\geq 0}(Q^kx)s^k;\quad Q^1(s)x=\sum_{k\geq 1}(\beta Q^kx)s^k.
\]
Then we obtain that $E^0(s)=Q^0(-s)[1]$ and $E^1(s)=Q^1(-s)[1]$. 

Let $s, t$ be formal variables. We remind the reader that if $I$ denote the string $(\epsilon_1, i_1, \ldots, \epsilon_n, i_n)$, then $b(I) = \epsilon_1 + \ldots + \epsilon_n$. The main result of this section is the following proposition. 
\begin{pro}\label{pro:relation of E(s)}
For $s,t$ are formal variables, we have relations
\begin{align}
E^0(s^{p-1})\circ E^0(t^{p-1})&=E^0(s^{p-1})\circ E^0((s
+t)^{p-1});\label{eq:relation E(s) 2}\\
E^0(s^{p-1})\circ E^1(t^{p-1})&=E^0(s^{p-1})\circ E^1((s+t)^{p-1})\frac{t}{s+t};\label{eq:relation E(s) 4}\\
E^1(s^{p-1})\circ E^1(t^{p-1})&=E^1(s^{p-1})\circ E^1((s+t)^{p-1})\frac{t}{s+t};\label{eq: relation E(s) 3}
\end{align}
 When $k=2i_1+b(I)+\epsilon_1$, $b(I)>0$ (or $b(I)>\epsilon_1$ if $k=0$), $\epsilon_n=1$, and $n>1$, 
\begin{equation}\label{eq:relation about excess_sigma}
\sigma^{\circ k}\circ E_{(\epsilon_1,i_1+b(I))}\circ\cdots\circ E_{(\epsilon_n,p^{n-1}(i_1+\cdots+i_n+b(I))-\Delta_n\epsilon_n)}=(1-\epsilon_1)y^{\star p}
\end{equation}
where $\sigma^{\circ 0}=[1]$, $y=(-1)^{ni_1+\cdots+i_n+nb(I)}\beta^{\epsilon_2}Q^{j_2}\cdots\beta^{\epsilon_n}Q^{j_n}(\sigma^{\circ k})\in H_*QS^k$, and $j_s$ is given in Lemma \ref{lm:relation E and Q};
\begin{align}
E_{(0,0)}&=[p];\label{eq:relation E(s) 1}\\
\sigma^{\circ 2k}\circ E_{(\epsilon,k)}&=(1-\epsilon)(-1)^k(\sigma^{\circ 2k})^{\star p}.\label{eq:relation of sigma}
\end{align}
\end{pro}
Lemma \ref{lm:relation E and Q} and the bijective correspondence in Lemma \ref{lm:bijective 2}  allow us to inductively express  the element $y$ in the right hand side of \eqref{eq:relation about excess_sigma}  (in fact, any element $Q^{(\varepsilon, J)} (\sigma^{\circ k}) \in H_* QS^k$)  in terms of circle products of the form $\sigma^{\circ k} \circ E_{(\epsilon_2, t_2)} \circ \cdots \circ  E_{(\epsilon_n, t_n)}$. Because of the excess condition, this inductive process must eventually stop. Thus \eqref{eq:relation about excess_sigma} indeed gives rise to a relation involving $\sigma$ and $E_{(\epsilon,i)}$'s.  

Before giving the proof of the proposition, we  make some remarks about relation \eqref{eq:relation about excess_sigma}.
\begin{rem}\label{rem:4.8}\rm
\begin{enumerate}
\item It is not too difficult to compute an explicit formula for $y$ when $n$ is small. For example, if $n=2$, then 
	\[
	y=(-1)^{2i_1+i_2+b(I)}E_{(\epsilon_2,i_1+i_2+b(I))}(\sigma^{\circ k}). 
	\] 
	For $n \geq 3$, the element $y$  may contain more than one summand. Here is an explicit example: If $p=3$ and $I=(0,2,1,3,1,3)$, then  $y = -\beta Q^{10p-4}\beta Q^{10}(\sigma^{\circ 6})$ and from Lemma \ref{lm:relation E and Q}, we get 
	\[
	y =\sigma^{\circ 6}\circ E_{(1,7)}\circ E_{(1,10p-1)}+\sigma^{\circ 6}\circ E_{(1,7-p)}\circ E_{(1,11p-1)}.
	\]

By an induction argument, one sees that $y$ is a sum of circle products of at most $(n-1)$ terms $E_{(\epsilon, i)}$. 
\item[(ii)]  The relation \eqref{eq:relation about excess_sigma} is still valid when $\epsilon_n=0$ but  no new relation will be obtained this way. Indeed, let  $s < n$ be such that $\epsilon_s=1$ and $\epsilon_i=0$ for $i>s$. Applying relation \eqref{eq:relation about excess_sigma} to the sequence $(\epsilon_1,i_1,\dots,\epsilon_s,i_s)$, we have
\[
\sigma^{\circ k}\circ E_{(\epsilon_1,i_1+b(I))}\circ\cdots\circ E_{(\epsilon_s,p^{s-1}(i_1+\cdots+i_s+b(I))-\Delta_s\epsilon_s)}=(1-\epsilon_1)y_1^{\star p}.
\]
By the distributivity law between the $\star$ product and the $\circ$ product (see \cite[Lemma 1.12]{Ravenel.Wilson1977}),  it follows that 
\begin{align*}
\sigma^{\circ k}\circ &E_{(\epsilon_1,i_1+b(I))}\circ\cdots\circ E_{(\epsilon_s,p^{s-1}(i_1+\cdots+i_s+b(I))-\Delta_s\epsilon_s)}\circ \cdots\circ E_{(0,p^{n-1}(i_1+\cdots+i_n+b(I))}\\
&=(1-\epsilon_1)y_1^{\star p}\circ E_{(0,p^{s}(i_1+\cdots+i_{s+1}+b(I))}\circ\cdots\circ E_{(0,p^{n-1}(i_1+\cdots+i_n+b(I))}\\
&=(1-\epsilon_1)(y_1\circ E_{(0,p^{s-1}(i_1+\cdots+i_{s+1}+b(I))}\circ\cdots\circ E_{(0,p^{n-2}(i_1+\cdots+i_n+b(I))})^{\star p}.
\end{align*}
To sum up,  \eqref{eq:relation about excess_sigma} applied to  a sequence $(\varepsilon, I)$ where $\epsilon_n = 0$ yields a relation which is a corollary of the corresponding relation for $(\epsilon_1,i_1,\dots,\epsilon_s,i_s)$  for some $s < n$ (in fact,
$s$ is the greatest number such that $\epsilon_s =1$) and the distributivity law.
\item[(iii)]  If $b(I)=0$ then \eqref{eq:relation about excess_sigma} can be derived from \eqref{eq:relation of sigma}.
However, when $b(I)>0$,  it is clear that \eqref{eq:relation about excess_sigma} is an independent relation. Indeed, considering the form of the relations, the only possible redundancy is between relations of the forms \eqref{eq:relation about excess_sigma} and 
\eqref{eq:relation of sigma}, and this occurs if and only if $k = 2i_1 + b(I) + \epsilon_1 = 2(i_1 + b(I))$ and $\epsilon_1 = 0$, so that $b(I) = 0$.

Finally, when $k=0$, if $b(I)=\epsilon_1=1$, then relation \eqref{eq:relation about excess_sigma} becomes to trivial relation; but if $b(I)>\epsilon_1=1$ or $\epsilon_1=0$, the relation is nontrivial.
\end{enumerate}
\end{rem}

\begin{rem}\rm
In the Bockstein-nil homology considered by Kashiwabara \cite{Kashiwabara2012}, the relations \eqref{eq:relation E(s) 4} and \eqref{eq: relation E(s) 3} do not show up because the Bockstein operation is involved. 

In the case $p=2$,  Turner's relation \cite{Tur97}  can be obtained from  \eqref{eq:relation E(s) 2} by halving the degree of $E_{(0,i)}$'s. Then, relations \eqref{eq:relation E(s) 4} and \eqref{eq: relation E(s) 3} become trivial under the halving degree map because the degree of $E_{(1,i)}$'s are odd. 
\end{rem}
\begin{proof}[of Proposition \ref{pro:relation of E(s)}]
It should be noted that the formulas \eqref{eq:relation E(s) 2}-\eqref{eq: relation E(s) 3} can be proved by using the method in Turner \cite{Tur97}, of course, it is more complicated.

We will make use two properties of  the transfer: it is multiplicative, and  $tr_*^{(n)}$ is $GL_n$-invariant. 

First, we can consider the first transfer $tr_*^{(1)}$ as the element
\[
tr_*^{(1)}\in \Hom(H_*B\V_1,H_*QS^0)\cong H^*(B\V_1,H_*QS^0)\cong H_*QS^0[[s]]\otimes E(\epsilon).
\]
Because $tr_*^{(1)}$ sends the generator in the degree $2(p-1)i$ to $E_{(0,i)}$, that in the degree $2(p-1)i-1$ to $E_{(1,i)}$, and the rest to zero, it is equal to
\[
E^0(s^{p-1})+\epsilon s^{-1}E^1(s^{p-1}).
\]

Next, the second transfer $tr_*^{(2)}$ can be considered as the element
\[
tr_*^{(2)}\in \Hom(H_*B\V_2, H_*QS^0)\cong H^*(B\V_2,H_*QS^0)\cong H_*QS^0[[s,t]]\otimes E(\epsilon,\epsilon').
\]
By the multiplicativity of the transfer, this element has to be
\[
(E^0(s^{p-1})+\epsilon s^{-1}E^1(s^{p-1}))\circ (E^0(t^{p-1})+\epsilon' t^{-1}E^1(t^{p-1})).
\]

Since the transfer factors through the coinvariant of $H_*B\V_2$ under the action of the general linear group ${GL_2}$, acting $\left(\begin{smallmatrix}1&1\\0&1\end{smallmatrix}\right)$ on $tr_*^{(2)}$, we obtain
\begin{align*}
&(E^0(s^{p-1})+\epsilon s^{-1}E^1(s^{p-1}))\circ (E^0(t^{p-1})+\epsilon' t^{-1}E^1(t^{p-1}))\\
&=(E^0(s^{p-1})+\epsilon s^{-1}E^1(s^{p-1}))\circ (E^0((s+t)^{p-1})+(\epsilon+\epsilon') (s+t)^{-1}E^1((s+t)^{p-1})).
\end{align*}
Expanding this equality and comparing the coefficients of ``$1$'', $\epsilon'$ and $\epsilon\epsilon'$ follows the formulas \eqref{eq:relation E(s) 2}, \eqref{eq:relation E(s) 4} and \eqref{eq: relation E(s) 3}.

From the above proof, we observe that the formulas \eqref{eq:relation E(s) 2}, \eqref{eq:relation E(s) 4} and \eqref{eq: relation E(s) 3} also hold in $colim_{\mathcal{B}\V /CS^0}H_*(-)[[s, t]]$, where $\mathcal{B}\V /CS^0$ is the category whose objects are homotopy classes of maps
from a classifying space of an elementary abelian $p$-group to $CS^0$, whose
morphism are commutative triangles, and $CS^0$ denotes the combinatorial
model of $QS^0$, that is, the disjoint union of $B\Sigma_n$'s (see \cite[Section 5]{Kashiwabara2012}).

From Lemma \ref{lm:relation E and Q}, for $n\geq 1$,
\begin{align*}
\sigma^{\circ k}\circ E_{(\epsilon_1,i_1+b(I))}\circ\cdots\circ& E_{(\epsilon_n,p^{n-1}(i_1+\cdots+i_n+b(I))-\Delta_n\epsilon_n)}\\
=(-1)^{ni_1+\cdots+i_n+nb(I)}&\beta^{\epsilon_1}Q^{j_1}\cdots\beta^{\epsilon_n}Q^{j_n}(\sigma^{\circ k})+\sum \alpha_KQ^K(\sigma^{\circ k}),
\end{align*}
where $exc(K)<exc(J)=2i_1+b(I)$.

Since $k=2i_1+b(I)+\epsilon_1$, the second sum of the formula is trivial.

If $\epsilon_1=1$, then $exc(J)<k$, therefore, the first term is also trivial. Otherwise, if $\epsilon_1=0$, then $2j_1=\deg (\beta^{\epsilon_2}Q^{j_2}\cdots\beta^{\epsilon_n}Q^{j_n}(\sigma^{\circ k}))$, hence, the first term is the $p$-th power of the element $y=(-1)^{ni_1+\cdots+i_n+nb(I)}\beta^{\epsilon_2}Q^{j_2}\cdots\beta^{\epsilon_n}Q^{j_n}$. Thus, the formula \eqref{eq:relation about excess_sigma} is proved.
\end{proof}

Since $\sigma$ is primitive elements, we have the following corollary. 
\begin{cor}\label{cor:relations}
For $n\geq1$ and $2i_1+b(I)<k$,
\[
\sigma^{\circ k}\circ E_{(\epsilon_1,i_1+b(I))}\circ\cdots\circ E_{(\epsilon_n,p^{n-1}(i_1+\cdots+i_n+b(I))-\Delta_n\epsilon_n)}=0,
\]
where $\sigma^{\circ0}=[1]$.
\end{cor}

Let us put $\underline{E}^\epsilon(s^{p-1})=s^{-1}E^\epsilon(s^{p-1})\in H_*QS^0[[s]]$, equalities \eqref{eq:relation E(s) 2}-\eqref{eq: relation E(s) 3} can be simplified as follows.
\begin{cor}\label{cor:reduced relation of E(s)}
For $s,t$ are formal variables, then we have relation
\begin{equation}\label{eq:reduced relation of E(s)}
\underline{E}^{\epsilon_1}(s^{p-1})\circ\underline{E}^{\epsilon_2}(t^{p-1})=\underline{E}^{\epsilon_1}(s^{p-1})\circ\underline{E}^{\epsilon_2}((s+t)^{p-1}),\quad \epsilon_1\leq \epsilon_2.
\end{equation}
\end{cor}

For $A\in GL_{\ell}, B\in GL_k$, denote $A\oplus B=\left(\begin{smallmatrix}A&0\\0&B\end{smallmatrix}\right)\in GL_{\ell+k}$ and $a\oplus A=\left(\begin{smallmatrix}a&0\\0&A\end{smallmatrix}\right)\in GL_{\ell+1}$. Then we have the lemma.
\begin{lem}\label{lm:generators of G}
For $n\geq 2$, the general linear group $GL_n=GL_n(\fp)$ is generated by $\{T, \Sigma_n, T_a:a\in \fp^*\}$, where
\[
T=\left(\begin{smallmatrix}1&1\\0&1\end{smallmatrix}\right)\oplus I_{n-2}, \quad T_a=a\oplus I_{n-1}.
\]
\end{lem}
\begin{thm}\label{thm:Hopf ring of QS0}
The homology $\{H_*QS^k\}_{k\geq0}$ is the coalgebraic ring in $\fp[\mathbb Z]$ generated by $E_{(0,i)} (i\geq0),$ $E_{(1,j)} (j\geq 1)$ and $\sigma$ modulo all the relations implied by Proposition \ref{pro:relation of E(s)}.

The coproduct is specified by
\[
\psi(\sigma)=1\otimes \sigma+\sigma\otimes 1;\quad \psi(E^0(s))=E^0(s)\otimes E^0(s);
\]
\[
\psi(E^1(s))=E^0(s)\otimes E^1(s)+E^1(s)\otimes E^0(s);\quad \psi(a\circ b)=\psi(a)\circ \psi(b).
\]
\end{thm}

The theorem can be proved by using the framework of Turner \cite{Tur97} and Eccles. et.al \cite{Eccles1997}, we mean that we can use the method in \cite{Tur97} to show for $k=0$, and then use the bar spectral sequence (as in \cite{Eccles1997}) to induct for $k>0$. Here, we modify the method of Turner \cite{Tur97} to show the theorem directly (without using the bar spectral sequence).
In order to do this, we need some notations.

We define elements $f^0(v_i,s), f^1(u_i,v_i,s)\in H_*B\V_n[[s]]$ for any $u_i,v_i$ by
\[
f^0(v_i,s)=\sum_{k\geq0}v_i^{[k]}s^k;\quad f^1(u_i,v_i,s)=\sum_{k\geq1}u_i v_i^{[k-1]}s^k,
\]
$f^0(0,s)=f^0(v_i,0)=1$. 

Then we have
\[
tr^{(1)}_*(f^0(v_j,s))=E^0(s^{p-1}); \quad tr^{(1)}_*(f^1(u_i,v_i,s))=E^1(s^{p-1}).
\]

Put $\underline{f}^0(v_i,s)=s^{-1}f^0(v_i,s)$ and $\underline{f}^1(u_i,v_i,s)=s^{-1}f^1(u_i,v_i,s)$, then
\[
tr^{(1)}_*(\underline{f}^0(v_j,s))=\underline{E}^0(s^{p-1}); \quad tr^{(1)}_*(\underline{f}^1(u_i,v_i,s))=\underline{E}^1(s^{p-1}).
\]
\begin{proof}
Let $D_{*,0}$ be the graded coalgebra over $\fp$ generated by $E_{(0,i)}\in D_{2i(p-1),0}\ (i\geq0)$, $E_{(1,j)}\in D_{2j(p-1)-1,0}\ (j\geq 1)$, where the coproduct is given as in the theorem. Let $D_{*,1}$ be the graded coalgebra over $\fp$ on the single primitive element $\sigma \in D_{1,1}$ and for $k\geq2$ let $D_{*,k}$ be the graded $\fp$ coalgebra with $D_{0,k}=\fp$ and zero in other dimensions. Actually, $D_{*,0}$ is isomorphic to the homology of $B\Sigma_p$ (see \cite{Adem.etal1990}) and $D_{*,1}$ is isomorphic to the homology of $S^1$. Apply the Ravenel-Wilson free Hopf ring functor \cite{Ravenel.Wilson1977} to the coalgebra $D_{*,*}$ to give $\mathscr{H}D_{*,*}$, the free $\fp[\mathbb{Z}]$-Hopf ring on $D_{*,*}$. There is a map of coalgebras $D_{*,*}\rightarrow \{H_*QS^k\}_{k\geq0}$ mapping the element $E_{(\epsilon,i)}$ to the element $E_{(\epsilon,i)}\in \{H_*QS^k\}_{k\geq0}$. By the universality, the map extends to a unique map of Hopf rings
\[
h \colon \mathscr{H}D_{*,*}\rightarrow \{H_*QS^k\}_{k\geq0}.
\]

Let $A_{*,*}$ be the free $\fp[\mathbb{Z]}$-Hopf ring on $D_{*,*}$ subject to relations arising from Proposition \ref{pro:relation of E(s)}. Since all relations defined in $A_{*,*}$ hold in $\{H_*QS^k\}_{k\geq0}$, $h$ induces a unique map
\[
\bar{h} \colon A_{*,*}\rightarrow \{H_*QS^k\}_{k\geq0}, 
\]
which is surjective by  Theorem \ref{thm:describe though dickson coninvariant}. Therefore, the induced map between  the indecomposables (with respect to $\star$ product) $QA_{*,k}\rightarrow QH_*QS^k$  is also a surjection.  We will prove that it is actually an isomorphism.

First of all, we claim that  for $k\geq0$, there exists an epimorphism $\bigoplus_{n\geq1}\mathscr{B}_k[n]^*\to A_{*,k}$. Indeed, let  $\underline{s}=(s_1,\dots,s_n)$ be an n-tuple of formal variables,  $\underline{\epsilon}=(\epsilon_1,\dots,\epsilon_n)$, with $\epsilon_i\in \{0,1\}$,  define
\[
\underline{u}^{\underline{\epsilon}} \; (\underline{s})=\underline{f}^{\epsilon_1}(u_1,v_1,s_1)\cdots \underline{f}^{\epsilon_n}(u_n,v_n,s_n)\in H_*B\V_n,
\]
where $\underline{f}^0(u_i,v_i,s_i)=\underline{f}^0(v_i,s_i)$. We also denote
\[
\underline{E}^{\underline{\epsilon}}(\underline{s}^{p-1})=\underline{E}^{\epsilon_1}(s_1^{p-1})\circ\cdots\circ \underline{E}^{\epsilon_n}(s_n^{p-1})\in A_{*,0}.
\]

Let $g:\bigoplus_{n\geq1}H_*B\V_n\rightarrow A_{*,0}$ be the map of $\fp$-algebras given by $\underline{u}^{\underline{\epsilon}}(\underline{s})\mapsto \underline{E}^{\underline{\epsilon}}(\underline{s}^{p-1})$. It is easy see that $g$ is a surjection. 

From Corollary \ref{cor:reduced relation of E(s)} and Lemma \ref{lm:generators of G}, it is easy to check that 
 $g(\underline{u}^{\underline{\epsilon}}A(\underline{s}))=\underline{E}^{\underline{\epsilon}}(\underline{s}^{p-1})=g(\underline{u}^{\underline{\epsilon}}(\underline{s}))$, for $A\in GL_n$. Therefore, $g$ factors through the coinvariants space of the general linear groups $\bigoplus_{n\geq1}(H_*B\V_n)_{GL_n}$.

Moreover, from Proposition \ref{pro:basis of B(p)}, elements $E_{(\epsilon_1,i_1,\dots,\epsilon_n,i_n)}\in (H_*B\V_n)_{GL_n}$, for $2i_1+b(I)<0$, represent the zero element in $\mathscr{B}_0[n]^*$. Therefore, from Theorem \ref{thm:basis in coinvariant} and Proposition \ref{pro:basis of H_GL}, they can be written as a combination of elements of the form 
$$[u_1^{\omega_1}v_1^{[(p-1)(j_1+\omega)-\omega_1]}\cdots u_n^{\omega_n}v_n^{[p^{n-1}(p-1)(j_1+\cdots+j_n+\omega)-p^{n-1}\omega_n]}],$$
for $\omega_i=0$ or $1$, $\omega=\omega_1+\cdots+\omega_n$ and $2j_1+\omega<0$.

Combining with the fact that $g$ is an algebra homomorphism, we get that the image of $E_{(\epsilon_1,i_1,\dots,\epsilon_n,i_n)}, 2i_1+b(I)<0$, under $g$ can be written as a combination of the elements of the form
$
E_{(\omega_1,j_1+\omega)}\circ\cdots\circ E_{(\omega_n,p^{n-1}(j_1+\cdots+j_n+\omega)-\Delta_n\omega_n)},
$
with $2j_1+\omega<0$, $\omega=\omega_1+\cdots+\omega_n$.
Therefore, from Corollary \ref{cor:relations}, one gets $g(E_{(\epsilon_1,i_1,\dots,\epsilon_n,i_n)})=0$ for $2i_1+b(I)<0$.

Hence, by Proposition \ref{pro:basis of B(p)}, $g$ factors through $\bigoplus_{n\geq1}\mathscr{B}_0[n]^*$. In other words, there exists homomorphism $\bar{g}:\bigoplus_{n\geq1}\mathscr{B}[n]^*\to A_{*,0}$ such that the diagram
\[
\bfig
\Vtriangle/>`>`<-/[\bigoplus_{n\geq1}H_*B\V_n`A_{*,0}`\bigoplus_{n\geq1}\mathscr{B}_0{[n]}^*;g`p`\overline{g}]
\efig
\]
is commutative.

For any $k\geq0$, let $g_k$ be the composition
\[
\bigoplus_{n\geq1}\mathscr{B}_0[n]\xrightarrow{\bar{g}}A_{*,0}\xrightarrow{\sigma^{\circ k}\circ -}A_{*,k}.
\]
When $k=0$, $g_0$ is just $\bar{g}$. Since, from Corollary \ref{cor:relations}, in $A_{*,k}$, $$\sigma^{\circ k}\circ E_{(\epsilon_1,i_1+b(I))}\circ\cdots\circ E_{(\epsilon_n,p^{n-1}(i_1+\cdots+i_n+b(I))-\Delta_n\epsilon_n)}=0, 2i_1+b(I)<k,$$ by the same above argument, the $\fp$-map $g_k$ factors through $\bigoplus_{n\geq1}\mathscr{B}_k[n]$ and $g_k$ is also a surjection.

For any $n\geq1$, let $QA_{*,k}[n]$ be the subspace of $QA_{*,k}$ spanned by all elements $\sigma^{\circ k}\circ E_{(\epsilon_1,i_1)}\circ\cdots\circ E_{(\epsilon_n,i_n)}$ and let $QH_*QS^k[n]$ be the subspace of $QH_*QS^k$ spanned by all elements $\beta^{\epsilon_1}Q^{j_1}\cdots\beta^{\epsilon_n}Q^{j_n}(\sigma^{\circ k})$.

By Theorem \ref{thm:basis in coinvariant}, in $\mathscr{B}_k[n]^*$, we have
\begin{align*}
{\rm Span}&\{E_{(\epsilon_1,i_1,\dots,\epsilon_n,i_n)}:2i_1+b(I)+\epsilon_1>k\}=\\
&{\rm Span}\{[u_1^{\epsilon_1}v_1^{[(p-1)(i_1+b(I))-\epsilon_1]}\cdots u_n^{\epsilon_n}v_n^{[p^{n-1}(p-1)(i_1+\cdots+i_n+b(I))-p^{n-1}\epsilon_n]}]:\\
&\quad\quad\quad\quad 2i_1+b(I)+\epsilon_1>k\}.
\end{align*}

Therefore, we have a surjection
\begin{align*}
S={\rm Span}&\{[u_1^{\epsilon_1}v_1^{[(p-1)(i_1+b(I))-\epsilon_1]}\cdots u_n^{\epsilon_n}v_n^{[p^{n-1}(p-1)(i_1+\cdots+i_n+b(I))-p^{n-1}\epsilon_n]}]:\\
&\quad\quad\quad\quad n\geq1,2i_1+b(I)+\epsilon_1>k\}\rightarrow QA_{*,k}[n].
\end{align*}
 It implies that, in each degree $d$, $\dim(S)\geq \dim(QA_{*,k}[n])\geq \dim(QH_*QS^k[n])$.
 
 Finally, we observe that, for each degree $d$,
 \begin{align*}
 {\rm Card}&\{(\epsilon_1,i_1,\dots,\epsilon_n,i_n)|2(i_1+\epsilon_1)(p^n-1)-\epsilon_1+\cdots\\
& +2(i_n+\epsilon_n)(p^n-p^{n-1})-\epsilon_n=d\}\\
={\rm Card}&\{(\epsilon_1,i_1,\dots,\epsilon_n,i_n)|\epsilon_1+2((p-1)(i_1+b(I))-\epsilon_1)+\cdots\\
&+\epsilon_n+2p^{n-1}((p-1)(i_1+\cdots+i_n+b(I))-\epsilon_n)=d\}.
 \end{align*}
 So $\dim(S)= \dim(QA_{*,k}[n])= \dim(QH_*QS^k[n])$. It implies  $QA_{*,k}\cong QH_*QS^k$.
 
 The proof is complete.
\end{proof}

\section{The actions of $A$ and $R$ on $H_*QS^k$}\label{sec: action}
The Hopf ring $\{H_*QS^k\}_{k\geq0}$ possesses a very rich structure, namely, it is an $A$-$R$-algebra, where $A$ is again the mod $p$ Steenrod algebra and $R$ is the mod $p$ Dyer-Lashof algebra. Further, the two actions are compatible in the sense made precise below. For convenience, we write $P^k$ instead of $P^k_*$ and write their action on the right. For $x\in H_*QS^k$ (or $x\in H_*B\V_n$) and formal variable $s$, we define the formal series
\[
xP^{\epsilon}(s)=\sum_{k\geq0}(x\beta^{\epsilon}P^k)s^k, \epsilon=0,1.
\]

In order to prove the main theorem of this section, we need the following lemma.
\begin{lem}\label{lm:relation P(s) and Q(s)}
For $x,y\in H_*QS^k$ and $u_i,v_i\in H_*B\V_n$ given in Section \ref{sec:Basis}, the following relations hold:
\begin{equation}\label{eq:P act on HBV 1}
x\circ Q^\epsilon(s)(y)=Q^\epsilon(s)(xP^0(s^{-1})\circ y)-\epsilon(-1)^{{\rm deg}y}Q^0(s)(xP^{1}(s^{-1})\circ y);
\end{equation}
\begin{equation}\label{eq: P act on HBV 2}
f^0(v_i,s)P^0(t)=f^0(v_i,(s+s^pt));
\end{equation}
\begin{equation}\label{eq: P act on HBV 3}
\underline{f}^0(v_i,s)P^1(t)=\underline{f}^1(u_i,v_i,(s+s^pt));
\end{equation}
\begin{equation}\label{eq: P act on HBV 4}
\underline{f}^1(u_i,v_i,s)P^0(t)=\underline{f}^1(u_i,v_i,(s+s^pt)).
\end{equation}
\end{lem}
\begin{proof}
From Theorem \ref{thm:action of DL on HQSn}, we obtain
{\allowdisplaybreaks
\begin{align*}
x\circ Q^\epsilon(s)&(y)\\
&=\sum_{k\geq\epsilon}\beta^\epsilon Q^{k+i}\left(\sum_{i\geq0}xP^i\circ y \right)s^k-\epsilon\sum_{k\geq\epsilon}(-1)^{{\rm deg}y}Q^{k+i}\left(\sum_{i\geq 1}x\beta P^i\circ y\right)s^k\\
&=\sum_{\ell\geq\epsilon+i}\beta^\epsilon Q^{\ell}\left(\sum_{i\geq0}xP^i\circ y \right)s^{\ell-i}-\epsilon\sum_{\ell\geq i}(-1)^{{\rm deg}y}Q^{\ell}\left(\sum_{i\geq 1}x\beta P^i\circ y\right)s^{\ell-i}
\end{align*}
}

It should be noted that if $xP^i$ (respect, $x\beta P^i$) is nontrivial then the degree of $xP^i\circ y$ (respect, $x\beta P^i\circ y$), is not less than $2i$ (respect, $2i+1$). It implies that, when $\ell<\epsilon+i$ (respect, $\ell<i$) then $\beta^\epsilon Q^{\ell}(xP^i\circ y)$ (respect, $Q^\ell(x\beta P^i\circ y)$) is trivial. 

Therefore, the right hand side of the above
formula can be written as follows
\begin{align*}
\sum_{\ell\geq\epsilon}\beta^\epsilon Q^{\ell}&\left(\sum_{i\geq0}(xP^i\circ y)s^{-i} \right)s^{\ell}-\epsilon\sum_{\ell\geq 0}(-1)^{{\rm deg}y}Q^{\ell}\left(\sum_{i\geq 1}(x\beta P^i\circ y)s^{-i}\right)s^{\ell}\\
&=Q^\epsilon(s)(xP^0(s^{-1})\circ y)-\epsilon(-1)^{{\rm deg}y}Q^0(s)(xP^{1}(s^{-1})\circ y).
\end{align*}
Hence, the formula \eqref{eq:P act on HBV 1} is proved.

From
\[
v_i^{[n]}\beta^\epsilon P^k=\binom{n-(p-1)k-\epsilon}{k}u_i^\epsilon v_i^{[n-(p-1)k-\epsilon]}
\]
we have the formulas \eqref{eq: P act on HBV 2} and \eqref{eq: P act on HBV 3}.

Since $P^k$ acts trivially on $u_i$ for $k>0$, then
\[
u_iv_i^{[n-1]}P^k=\binom{n-(p-1)k-1}{k}u_iv_i^{[n-(p-1)k-1]}.
\]
This implies the last formula.
\end{proof}

The main results of the section is the following theorem, which gives a description of the actions of the Dyer-Lashof algebra and the Steenrod algebra on the Hopf ring.
\begin{thm}[{(Compare with \cite[Theorem 5.1]{Tur97})}]\label{thm:action of R on Hopf ring}
Let $x,y\in H_*QS^k$ and let $s,t,t_1,t_2,\dots$ be formal variables; $\underline{t}^{p-1}=(t_1^{p-1},\dots,t_n^{p-1})$, $\underline{\epsilon}=(\epsilon_1,\dots,\epsilon_n)$. The following hold in $H_*QS^k[[s,t,t_1,t_2,\dots]]$.
{\allowdisplaybreaks
\begin{align}
[n] P^\epsilon(s)&=(1-\epsilon)[n].\label{eq:P act on HQS 1}\\
E^0(s^{p-1})P^0(t)&=E^0((s+s^pt)^{p-1}). \label{eq:P act on HQS 2}\\
\underline{E}^0(s^{p-1})P^1(t)&=\underline{E}^1((s+s^pt)^{p-1}). \label{eq:P act on HQS 2'}\\
\underline{E}^1(s^{p-1})P^0(t)&=\underline{E}^1((s+s^pt)^{p-1}).\label{eq:P act on HQS 3}\\
E^1(s^{p-1})P^1(t)&=0.\label{eq:P act on HQS 3'}\\
(x\star y)P^\epsilon(s)&=(-1)^{\epsilon\deg y}xP^\epsilon(s)\star yP^0(s)+\epsilon (xP^0(s))\star yP^1(s).\label{eq:P-cartan star}\\
(x\circ y)P^\epsilon(s)&=(-1)^{\epsilon\deg y}xP^\epsilon(s)\circ yP^0(s)+\epsilon(xP^0(s))\circ yP^1(s).\label{eq:P-cartan circ}\\
Q^\epsilon(s)[n]&=[n]\circ E^\epsilon(-s).\label{eq:Q act on HQS 1}\\
Q^{\epsilon_1}(s^{p-1})E^{\epsilon_2}((st)^{p-1})&= (1+\hat{t}^{p-1})[E^{\epsilon_2}\left((s\hat{t})^{p-1}\right)\circ E^{\epsilon_1}(-s^{p-1})\notag\\
&\quad\quad+\epsilon_1(1-\epsilon_2) E^1\left((s\hat{t})^{p-1}\right)\circ E^0(-s^{p-1})].\label{eq: Q act on HQS 2}\\
Q^\epsilon(s)(x\star y)&=Q^\epsilon(s)x\star Q^0(s)y+\epsilon(-1)^{\epsilon\deg y} Q^0(s)x\star Q^1(s)y.\label{eq: Q cartan star}\\
Q^\epsilon(s)([n]\circ y)&=[n]\circ Q^\epsilon(s)y.\label{eq: Q cartan circ}\\
Q^{\epsilon}(s^{p-1})(E^{\underline{\epsilon}}((s\underline{t})^{p-1}))&=\ (1+\underline{\hat{t}}^{p-1})[E^{\underline{\epsilon}}((s\hat{\underline{t}})^{p-1})\circ E^{\epsilon}(-s^{p-1})\notag\\
&\quad+\epsilon\sum_{i=1}^n(1-\epsilon_i) E^{\underline{\epsilon}_i}((s\hat{\underline{t}})^{p-1})\circ E^{0}(-s^{p-1})].
 \end{align}
 }
 Here we denote by $\hat{t}=\sum_{k\geq0}(-1)^kt^{p^k}$, $\hat{t_i}=\sum_{k\geq0}(-1)^kt_i^{p^k}$, $\underline{\hat{t}}^{p-1}=(\hat{t}_1^{p-1},\dots,\hat{t}_n^{p-1})$, and $\underline{\epsilon}_i$ the vector obtained from $\underline{\epsilon}$ by replacing $\epsilon_i$ by $1$.
\end{thm}
\begin{proof}
The first equality is immediate by degree.

Since $tr^{(1)}_*(v^{[n]}P^k)=tr^{(1)}_*(v^{[n]})P^k$, then equalities \eqref{eq:P act on HQS 2}-\eqref{eq:P act on HQS 3'} are implied from \eqref{eq: P act on HBV 2}, \eqref{eq: P act on HBV 3} and \eqref{eq: P act on HBV 4}.

Since the coproduct of $P^{\epsilon}(s)$ is given by
\[
\psi(P^\epsilon(s))=P^\epsilon(s)\otimes P^0(s)+\epsilon P^0(s)\otimes P^1(s),
\]
the formulas \eqref{eq:P-cartan star} and \eqref{eq:P-cartan circ} come from the Cartan formula.

Letting $y=[1]$ in \eqref{eq:P act on HBV 1} to obtain
\begin{equation}\label{eq:replace y=[1]}
x\circ Q^\epsilon(s)[1]=Q^\epsilon(s)(xP^0(s^{-1}))-\epsilon Q^0(s)(xP^1(s^{-1})).
\end{equation}

Letting $x=[n]$ in the above equality and combining with \eqref{eq:P act on HQS 1}, we obtain \eqref{eq:Q act on HQS 1}.

Replace $x=E^{\epsilon'}(u^{p-1})$ in \eqref{eq:replace y=[1]}, we get
\begin{align*}
E^{\epsilon'}(u^{p-1})\circ Q^\epsilon(s)[1]&=Q^\epsilon(s)(E^{\epsilon'}(u^{p-1})P^0(s^{-1}))\\
&\quad\quad-\epsilon Q^0(s)(E^{\epsilon'}(u^{p-1})P^1(s^{-1})).
\end{align*}
Combining with \eqref{eq:P act on HQS 2}-\eqref{eq:P act on HQS 3'}, we give
\begin{align}
E^{\epsilon'}(u^{p-1})\circ Q^\epsilon(s)[1]&=(1+u^{p-1}t^{-1})^{-\epsilon'}
[Q^\epsilon(s)(E^{\epsilon'}(u+u^ps^{-1})^{p-1})\notag\\
&\quad\quad-\epsilon(1-\epsilon')Q^0(s)(E^{1}(u+u^ps^{-1})^{p-1})].\label{eq:actions}
\end{align}
From \eqref{eq:actions}, letting $\epsilon=0$ and $\epsilon'=1$, one gets
\[
E^1(u^{p-1})\circ Q^0(s)[1]=(1+u^{p-1}t^{-1})^{-1}
Q^0(s)(E^{1}(u+u^ps^{-1})^{p-1}).
\]

These formulas imply (replacing $s$ by $s^{p-1}$)
\begin{align*}
Q^{\epsilon_1}(s^{p-1})&E^{\epsilon_2}((u+u^ps^{1-p})^{p-1})\\
&=(1+u^{p-1}s^{1-p})[E^{\epsilon_2}(u^{p-1})\circ Q^{\epsilon_1}(s^{p-1})[1]+\epsilon_1(1-\epsilon_2)E^1(u^{p-1})\circ Q^0(s^{p-1})[1]].
\end{align*}
By letting $t=u/s+(u/s)^p$ with noting that $\hat{t}=\sum_{k\geq0}(-1)^kt^{p^k}=u/s$, it is easy to write the equality in the form
\begin{align*}
Q^{\epsilon_1}(s^{p-1})&E^{\epsilon_2}((st)^{p-1})\\
=&\ (1+\hat{t}^{p-1})[E^{\epsilon_2}\left((s\hat{t})^{p-1}\right)\circ E^{\epsilon_1}(-s^{p-1})+\epsilon_1(1-\epsilon_2) E^1\left((s\hat{t})^{p-1}\right)\circ E^0(-s^{p-1})].
\end{align*}
So \eqref{eq: Q act on HQS 2} is proved. The equality \eqref{eq: Q cartan star} is just the Cartan formula.

In order to prove formula \eqref{eq: Q cartan circ}, we can use \eqref{eq:P act on HBV 1} and \eqref{eq:P act on HQS 1}, but here we use the distributivity law and the Cartan formula. From \cite{Ravenel.Wilson1977}, we have
\begin{align*}
Q^0(s)([n]\circ y)&=Q^0(s)\left(\sum y'\star\cdots\star y^{(n)}\right)\\
&=\sum (Q^0(s)y')\star\cdots\star(Q^0(s)y^{(n)})=[n]\circ Q^0(s)y;
\end{align*}
and
\begin{align*}
Q^1(s)([n]\circ y)&=Q^1(s)\left(\sum y'\star\cdots\star y^{(n)}\right)\\
&=\sum(Q^1(s)y')\star\cdots\star(Q^0(s)y^{(n)})+\cdots+\sum(Q^0(s)y')\star\cdots\star(Q^1(s)y^{(n)})\\
&=[n]\circ Q^1(s)y.
\end{align*}
Here $(n-1)$-fold coproduct of $y$ is $\sum y'\otimes \cdots\otimes y^{(n)}$.
Thus, \eqref{eq: Q cartan circ} is proved.

Since $(n-1)$-fold coproduct of $P^\epsilon(s)$ is given by
\[
\psi^{n-1}(P^0(s))=P^0(s)\otimes \cdots\otimes P^0(s),
\]
and
\[
\psi^{n-1}(P^1(s))=P^1(s)\otimes \cdots\otimes P^0(s)+\cdots+P^0(s)\otimes \cdots\otimes P^1(s),
\]
the last formula follows from formula \eqref{eq: Q act on HQS 2} and the Cartan formula.

The proof is complete.
\end{proof}
As noted in the introduction, the two categories of $A$-$H_*QS^0$-coalgebraic modules and $A$-$R$-allowable Hopf algebras seems to be fundamental  in the study of the mod $p$ homology of the infinite loop spaces. We will investigate these categories and their relationship  elsewhere.
\subsection*{Acknowledgements} The author would like to thank L\^e Minh H\`a and J. Peter May for many fruitful discussion, and Takuji Kashiwabara for sharing his insight and for many comments on an earlier version of this paper. The author is grateful to the referee for helpful comments and corrections. The paper was completed while the author was visiting the Vietnam Institute for Advanced Study in Mathematics (VIASM). He thanks the VIASM for support and hospitality.

\bibliographystyle{amsplain}
%\bibliography{ChonsBibliography}
\providecommand{\bysame}{\leavevmode\hbox to3em{\hrulefill}\thinspace}
\providecommand{\MR}{\relax\ifhmode\unskip\space\fi MR }
\providecommand{\MRhref}[2]{%
  \href{http://www.ams.org/mathscinet-getitem?mr=#1}{#2}
}
\providecommand{\href}[2]{#2}

%\address{Phan Ho{\`a}ng Ch\horn{o}n\\ 
%Department of Mathematics and Application, 
%Saigon University, 
%273 An Duong Vuong, District 5, Ho Chi Minh city, Vietnam.
% \email{chonkh@gmail.com}
 \end{document}